\documentclass[10pt,letterpaper]{article}
\pdfoutput=1
\usepackage{lineno}
\usepackage{graphicx,float,subfigure}
\usepackage{multicol,amsmath,amsfonts,amsthm,amssymb,array}
\usepackage{hyperref}
\usepackage{subfigure}
\usepackage{stmaryrd}
\usepackage{amsmath}
\usepackage{amsfonts}
\usepackage{amssymb}
\usepackage{graphicx}%
\usepackage{appendix}
\providecommand{\U}[1]{\protect\rule{.1in}{.1in}}
\hypersetup{urlcolor=blue, colorlinks=true, linkcolor=black,citecolor=black}
\newtheorem{theorem}{Theorem}

\newtheorem{corollary}[theorem]{Corollary}

\newtheorem{example}[theorem]{Example}

\newtheorem{lemma}[theorem]{Lemma}

\newcommand{\sgn}{\text{sgn}}
\begin{document}
\begin{center}  {Stability and bifurcation analysis of a SIR model with saturated incidence rate and saturated treatment}
\end{center}

\smallskip

\smallskip
\begin{center}
{\small \textsc{Erika Rivero-Esquivel}\footnote{email: erika.rivero@correo.uady.mx}, \textsc{Eric \'Avila-Vales}\footnote{email: avila@uady.mx}, \textsc{Gerardo E. Garc\'ia-Almeida}\footnote{Corresponding author. email: galmeida@uady.mx Tel. (999) 942 31 40 Ext. 1108.}
}
\end{center}
\begin{center} {\small \sl $^{1,2,3}$ Facultad de Matem\'aticas, Universidad Aut\'onoma de Yucat\'an, \\ Anillo Perif\'erico Norte, Tablaje 13615, C.P. 97119, M\'erida, Yucat\'an, Mexico}
\end{center}

\bigskip

{\small \centerline{\bf Abstract}
\begin{quote}
We study the dynamics of a SIR epidemic model with nonlinear incidence rate, vertical transmission vaccination for the newborns and the capacity of treatment, that takes into account the limitedness of the medical resources and the efficiency of the supply of available medical resources. Under some conditions we prove the existence of backward bifurcation, the stability and the direction of Hopf bifurcation. We also explore how the mechanism of backward bifurcation affects the control of the infectious disease. Numerical simulations are presented  to illustrate the theoretical findings.
\end{quote}}

\medskip


\bigskip

\section{Introduction}

Mathematical models that describe the dynamics of infectious diseases in
communities, regions and countries can contribute to have better approaches in
the disease control in epidemiology. Researchers always look for thresholds,
equilibria, periodic solutions, persistence and eradication of the disease.
For classical disease transmission models, it is common to have one endemic
equilibrium and that the basic reproduction number tells us that a disease is
persistent if is greater than 1,and dies out if is less than 1. This kind of
behaviour associates to forward bifurcation. However, there are epidemic
models with multiple endemic equilibria
\cite{hadeler1997,dushoff1998,driessche2000,brauer2004}, within these models
it can happen that a stable endemic equilibrium coexist with a disease free
equilibrium, this phenomenon is called backward bifurcation \cite{hadeler1995}.

In order to prevent and control the spread of infectious diseases like,
measles, tuberculosis and influenza, treatment is an important and effective
method. In classical epidemic models, the treatment rate of the infectious is
assumed to be proportional to the number of the infective individuals
\cite{anderson1991}. Therefore we need to investigate how the application of
treatment affects the dynamical behaviour of these diseases. In that direction
in \cite{wang2004}, Wang and Ruan, considered the removal rate
\[
T(I)=\left\{
\begin{array}
[c]{ll}%
k, & \text{if }I>0\\
0, & \text{if }I=0.
\end{array}
\right.
\]

In the following model
\[
\begin{aligned}
\frac{dS}{dt}&= A-dS-\lambda SI, \\
\frac{dI}{dt}&=\lambda SI-(d+\gamma)I-T(I), \\
\frac{dR}{dt}&=\gamma I+T(I)-dR,
\end{aligned}
\]
where S, I , and R denote the numbers of the susceptible, infective and
recovered individuals at time t , respectively. The authors study the
stability of equilibria and prove the model exhibits Bogdanov-Takens
bifurcation, Hopf bifurcation and Homoclinic bifurcation. In \cite{zhangliu},
the authors introduce a saturated treatment
\[
T(I)=\frac{\beta I}{1+\alpha I}.
\]
A related work is \cite{zhoufan}, \cite{wancui}. \par
 Hu, Ma and Ruan \cite{humaruan} studied the model
\begin{equation}
\begin{aligned} \dfrac{dS}{dt} =& bm(S+R) - \frac{\beta SI}{1+\alpha I} -bS+p\delta I\\ \dfrac{dI}{dt} =& \frac{\beta SI}{1+\alpha I}+(q\delta -\delta -\gamma)I - T(I)\\ \dfrac{dR}{dt} =& \gamma I-bR+bm'(S+R)+T(I) \label{hmr1}
 \end{aligned}
\end{equation}
the basic assumptions for the model \eqref{hmr1} are, the total population
size at time $t$ is denoted by $N=S+I+R$. The newborns of $S$ and $R$ are
susceptible individuals, and the newborns of $I$ who are not vertically
infected are also susceptible individuals, $b$ denotes the death rate and
birth rate of susceptible and recovered individuals, $\delta $ denotes the death
rate and birth rate of infective individuals, $\gamma $ is the natural recovery rate
of infective individuals. $q$ ($q\leq1$) is the vertical transmission rate,
and note $p=1-q$, then $0\leq p\leq1$. Fraction $m^{\prime}$ of all newborns
with mothers in the susceptible and recovered classes are vaccinated and
appeared in the recovered class, while the remaining fraction, $m=1-m^{\prime
}$, appears in the susceptible class, the incidence rate is described by a
nonlinear function $\beta SI/(1+\alpha I)$, where $\beta$ is a positive
constant describing the infection rate and $\alpha$ is a nonnegative constant.
The treatment rate of the disease is
\[
T(I)=\left\{
\begin{array}
[c]{ll}%
kI, & \text{if }0\leq I\leq I_{0},\\
u=kI_{0}, & \text{if }I>I_{0}%
\end{array}
\right.
\]
where $I_{0}$ is the infective level at which the healthcare systems reaches
capacity.\newline
In this work we will extend model \eqref{hmr1} introducing  the treatment rate
 $\frac{\beta_{2} I}{1+\alpha_{2} I}$, where $\alpha_{2}$, $\beta_{2}>0$,
obtaining the following model
\begin{equation}
\label{hmr3}\begin{aligned} \dfrac{dS}{dt} =& bm(S+R) - \frac{\beta SI}{1+\alpha I} -bS+p\delta I\\
 \dfrac{dI}{dt} =& \frac{\beta SI}{1+\alpha I}+(q\delta -\delta -\gamma)I - \frac{\beta_2 I}{1+\alpha_2 I}\\
  \dfrac{dR}{dt} =& \gamma I-bR+bm'(S+R)+\frac{\beta_2 I}{1+\alpha_2 I}. \end{aligned}
\end{equation}
Because $\frac{dN}{dt}=0$, the total number of population $N$ is constant. For
convenience, it is assumed that $N=S+I+R=1$. By using $S+R=1-I$, the first two
equations of \eqref{hmr3} do not contain the variable $R$. Therefore, system
\eqref{hmr3} is equivalent to the following 2-dimensional system:
\begin{equation}
\label{ruanmod}\begin{aligned} \dfrac{dS}{dt} &=-\dfrac{\beta SI}{1+\alpha I}-bS+bm(1-I)+p\delta I \\ \dfrac{dI}{dt} &=\dfrac{\beta SI}{1+\alpha I}-p\delta I -\gamma I-\dfrac{\beta _2 I}{1+\alpha _2 I} . \end{aligned}
\end{equation}

The parameters in the model are described below:

\begin{itemize}
\item $S,I,R$ are the normalized susceptible, infected, and recovered
population, respectively, therefore it follows that $S,I,R\leq1.$

\item $b$ is a positive number representing the birth and death rate of
susceptible and recovered population.

\item $\delta$ is a positive number representing the birth and death rate of
infected population.

\item $\gamma$ is a positive number giving the natural recovery rate of
infected population.

\item $q$ is positive ($q\leq1$) representing the vertical transmission rate
(disease transmission from mother to son before or during birth). It is
assumed that descendents of the susceptible and recovered classes belong to
the susceptible class, in the same way to the fraction of the newborns of the
infected class not affected by vertical transmission.

\item $p=1-q$ therefore $0\leq p\leq1$.

\item $m^{\prime}$ is positive and it is the fraction of vaccinated newborns
from susceptible and recovered mothers and therefore belong to the recovered
class. $m=1-m^{\prime}\geq0$ is the rest of newborns, which belong to the
susceptible class.

\item $\beta$ is positive, representing the infection rate, $\alpha$ is a
positive saturation constant (In the model the incidence rate is given by the
nonlinear function $\frac{\beta SI}{1+\alpha I}$ ).

\item $\frac{\beta_{2}I}{1+\alpha_{2}I}$ is the treatment function, satisfying
$\lim_{I\rightarrow\infty}\frac{\beta_{2}I}{1+\alpha_{2}I}=\frac{\beta_{2}%
}{\alpha_{2}},$ where $\alpha_{2},\beta_{2}>0.$
\end{itemize}

We note that if $\alpha_{2}=0$ the treatment becomes bilinear, case considered
in \cite{humaruan}, whereas if $\beta_{2}=0$ treatment is null, not being of
interest here. Therefore we will assume $\beta_{2},\alpha_{2}>0.$ \par

The paper in distributed as follows: in section 2 we compute the equilibria points and determine the conditions of its existence (as real values) and positivity, in section 3 we analyze the stability of the disease free equilibrium and endemic equilibria points in terms of value of $\mathcal{R}_{0} $ and the parameters of treatment function. Section 4 is dedicated to study Hopf bifurcation of the endemic equilibria points and section 5 shows discussion of all our results and we give some control measures that could be effective to eradicate the disease in each case. \par

Following \cite{humaruan} we define
\begin{equation}
\mathcal{R}_{0}:=\dfrac{\beta m}{\beta_{2}+p\delta+\gamma}.
\end{equation}
When $ \beta_2 =0 $, $\mathcal{R}_{0} $ reduces to
\begin{equation}
\mathcal{R}_{0} ^{*} = \dfrac{ \beta m}{ p \delta + \gamma },
\end{equation}
which is the basic reproduction number of  model \eqref{ruanmod} without treatment.
\begin{lemma}
Given the initial conditions $S(0)=S_{0}>0,I(0)=I_{0}>0$, then the solution of
(\ref{ruanmod}) satisfies $S(t),I(t)>0\quad\forall t>0$ and $S(t)+I(t)\leq1$.
\label{teo1}
\end{lemma}

\begin{proof}
Take the solution $S(t),I(t)$ satisfying the initial conditions $S(0)=S_{0}%
>0,I(0)=I_{0}>0$. Assume that the solution is not always positive, i.e., there
exists a $t_{0}$ such that $S(t_{0})\leq0$ or $I(t_{0})\leq0$. By Bolzano's
theorem there exists a $t_{1}\in(0,t_{0}]$ such that $S(t_{1})=0$ or
$I(t_{1})=0$, which can be written as $S(t_{1})I(t_{1})=0$ for some $t_{1}%
\in(0,t_{0}]$. Let
\begin{equation}
t_{2}=\min\{t_{i},S(t_{i})I(t_{i})=0\}.
\end{equation}
Assume first that $S(t_{2})=0$, then $\frac{dS\left(  t_{2}\right)  }{dt}>0$
implying that $S$ is increasing at $t=t_{2}.$ Hence $S(t)$ is negative for
values of $t<t_{2}$ near $t_{2},$ a contradiction. Therefore $S(t)>0$ $\forall
t>0$ and we must have $I(t_{2})=0,$ implying $\frac{dI(t_{2})}{dt}=0.$ Note
that if for some $t\geq0$ $I(t)=0,$ then $\frac{dI(t)}{dt}=0.$ Then any
solution with $I(0)=I_{0}=0$ will satisfy $I(t)=0$ $\forall t>0.$ By
uniqueness of solutions this fact implies that if $I(0)=I_{0}>0,$ then $I(t)$
will remain positive for all $t>0.$ Therefore $I(t_{2})=0$ leads to a
contradiction. Hence both $S$ and $I$ are nonnegative for all $t>0.$ Finally,
adding both derivatives of $S(t)$ and $I(t)$ we get:
\begin{equation}
\dfrac{d(S+I)}{dt}=-bS+bm-bmI-\gamma I-\dfrac{\beta_{2}I}{1+\alpha_{2}I}%
\end{equation}
Being $S,I\geq0$, if $S+I=1$ then $0\leq S\leq1$, $0\leq I\leq1$. Analyzing
the expression $-bS+bm-bmI$,
\[
-bS+bm-bmI=b(m-mI-S)=b(m-mI-1+I)=b(m-1+I(1-m)).
\]
Note that by the definition of the model parameters, $1-m=m^{\prime}\geq0$.
Knowing that $I\leq1,$ then
\begin{equation}
I(1-m)\leq1-m\Rightarrow I(1-m)+m-1\leq0.
\end{equation}
Therefore $-bS+bm-bmI\leq0$. Hence $\frac{d(S+I)}{dt}\leq0$ and $S+I$ is non
increasing along the line $S+I=1$, implying that $S+I\leq1$. Note also that
$S+I$ cannot be grater than 1, otherwise from $R=1-(S+I),$ $R$ would be
negative, a nonsense.
\end{proof}

\section{Existence and positivity of equilibria}

Assume that system (\ref{ruanmod}) has a constant solution $(S_{0},I_{0})$, then:%

\begin{align}
-\dfrac{\beta S_{0}I_{0}}{1+\alpha I_{0}}-bS_{0}+bm(1-I_{0})+p\delta I_{0}  &
=0\label{eqec1}\\
\dfrac{\beta S_{0}I_{0}}{1+\alpha I_{0}}-p\delta I_{0}-\gamma I_{0}%
-\dfrac{\beta_{2}I_{0}}{1+\alpha_{2}I_{0}}  &  =0 \label{eqec2}%
\end{align}
From (\ref{eqec1}) we obtain
\begin{equation}
 S_{0}=\dfrac{(1+\alpha I_{0})(bm(1-I_{0})+p\delta I_{0})}{\beta
I_{0}+b(1+\alpha I_{0})}. \label{eqec3}%
\end{equation}
And we get from (\ref{eqec2}):
\begin{align}
& I_{0}\left(  \dfrac{\beta S_{0}}{1+\alpha I_{0}}-p\delta-\gamma-\dfrac
{\beta_{2}}{1+\alpha_{2}I_{0}}\right)   =0\nonumber\\
& \Rightarrow I_{0}=0\quad  \text{or}\quad\dfrac{\beta S_{0}}{1+\alpha I_{0}%
}-p\delta-\gamma-\dfrac{\beta_{2}}{1+\alpha_{2}I_{0}}=0. \label{eqec4}%
\end{align}
If $I_{0}=0$ then $S_{0}=m$, obtaining in that way the disease-free
equilibrium $E=(m,0)$.

\begin{theorem}
System (\ref{ruanmod}) has a positive disease-free equilibrium $E=(m,0)$.
\end{theorem}

In order to obtain positive solutions of system \ref{ruanmod}  if $I_{0}\neq0$ then:%

\begin{align}
& \dfrac{\beta S_{0}}{1+\alpha I_{0}}-p\delta-\gamma-\dfrac{\beta_{2}}%
{1+\alpha_{2}I_{0}}    =0\nonumber\\
& \Rightarrow S_{0}    =\dfrac{1+\alpha I_{0}}{\beta}\left(  p\delta
+\gamma+\dfrac{\beta_{2}}{1+\alpha_{2}I_{0}}\right)  \label{eqec5}%
\end{align}

We obtain the following quadratic equation:
\begin{equation}
AI_{0}^{2}+BI_{0}+C=0. \label{eqec7}%
\end{equation}
Or
\begin{equation}
I_{0}^{2}+(B/A)I_{0}+C/A=0, \label{eqec8}%
\end{equation}
where the coefficients are given by:
\begin{align}
A  &  =\alpha_{2}(\beta(\gamma+bm)+\alpha b(p\delta+\gamma))>0\nonumber\\
B  &  =\beta(\gamma+\beta_{2}+bm(1-\alpha_{2}))+b\alpha(p\delta+\gamma
+\beta_{2})+b\alpha_{2}(p\delta+\gamma)\nonumber\\
&  =\beta(\gamma+\beta_{2}+bm-bm\alpha_{2})+b\alpha(1-\mathcal{R}_{0}%
)(p\delta+\gamma+\beta_{2})+\beta mb\alpha+b\alpha_{2}(p\delta+\gamma
)\nonumber\\
C  &  =b(p\delta+\gamma+\beta_{2}-\beta m)=b(p\delta+\gamma+\beta
_{2})(1-\mathcal{R}_{0}).\label{abc}
\end{align}
Its roots are:
\begin{align}
I_{1}  &  =\dfrac{-B-\sqrt{B^{2}-4AC}}{2A}\nonumber\\
I_{2}  &  =\dfrac{-B+\sqrt{B^{2}-4AC}}{2A}. \label{I0}%
\end{align}

And using these values in (\ref{eqec5}) we obtain its respective values
\begin{align}
S_{1}  &  =\dfrac{1+\alpha I_{1}}{\beta}\left(  p\delta+\gamma+\dfrac
{\beta_{2}}{1+\alpha_{2}I_{1}}\right) \nonumber\\
S_{2}  &  =\dfrac{1+\alpha I_{2}}{\beta}\left(  p\delta+\gamma+\dfrac
{\beta_{2}}{1+\alpha_{2}I_{2}}\right)  . \label{eqec14}%
\end{align}
Then our candidate for endemic equilibria are $E_{1}=(S_{1},I_{1})$, $E_{2}%
=(S_{2},I_{2})$.

Note that $C=0$ if and only if $\mathcal{R}_{0}=1$, $C>0$ if and only if
$\mathcal{R}_{0}<1,$ and $C<0$ if and only if $\mathcal{R}_{0}>1$ .

\begin{figure}[t]
\centering
\includegraphics[scale=0.6]{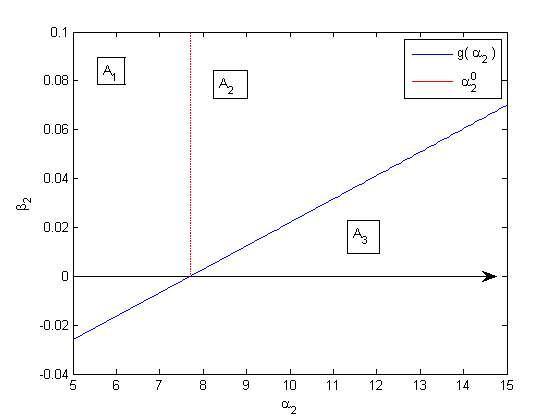}\caption{Location of the sets
$A_{1},A_{2},A_{3}$ in the plane $\alpha_{2} - \beta_{2}$, for
$\gamma=0.01,\beta=0.2,b=0.2,m=0.3,p=0.02,\delta=0.1$ }%
\label{fig1}%
\end{figure}

For $ \mathcal{R}_{0} ^{*}>1 $ we define the following sets:
\begin{align}
A_{1}  &  =\{(\beta_{2},\alpha_{2}): \beta_{2}>0,0<\alpha_{2}\leq \alpha_2^{0} ,\nonumber\\
A_{2}  &  =\{(\beta_{2},\alpha_{2}): \beta_{2}\geq g(\alpha_{2}),\alpha
_{2}> \alpha_2^{0}%
>0\},\nonumber\\
A_{3}  &  =\{(\beta_{2},\alpha_{2}): 0<\beta_{2}<g(\alpha_{2}),\alpha_{2}%
> \alpha_2^{0} >0\}.
\end{align}
Where
$$ \alpha_2^{0} = \frac
{-\beta(mb\alpha+\gamma+bm)}{b(p\delta+\gamma-\beta m)}$$
\[
g(\alpha_{2})=-\frac{1}{\beta}(b\alpha_{2}(p\delta+\gamma-\beta m)+\beta
(\gamma+bm+mb\alpha)).
\]
Define :%

\begin{align}
P_{1}  &  = 1 + \dfrac{1}{b \alpha( p \delta+ \gamma+ \beta_{2})} [ \beta(
\gamma+ \beta_{2} + bm - bm \alpha_{2} ) + \beta m b \alpha+ b \alpha_{2} ( p
\delta+ \gamma) ]\nonumber\\
R_{0}^{+}  &  = 1 - \dfrac{1}{b \alpha^{2} ( p \delta+ \gamma+ \beta_{2}%
)}\nonumber\\
&  \left[  \sqrt{- \beta\alpha( bm \alpha+ \beta_{2} + \gamma+ bm - \alpha
_{2}bm ) + \beta\alpha_{2} ( \gamma+ bm ) } - \sqrt{ \alpha_{2} ( \beta\gamma+
\beta b m + \alpha b p \delta+ \alpha b \gamma) } \right]  ^{2} .
\end{align}

Figure \ref{fig1} shows the location of these sets.

\begin{theorem}
If $\mathcal{R}_{0}>1$  the system (\ref{ruanmod}) has a unique (positive)
endemic equilibrium $E_{2}$. \label{teo10}
\end{theorem}

\begin{proof}
If $\mathcal{R}_{0}>1$ then $C<0$, then using Routh Hurwitz criterion for $n=2$, the
quadratic equation has two real roots with different sign, $I_{1}$ and $I_{2}%
$, where $I_{1}<I_{2}$. Hence there exists a unique
positive endemic equilibrium $E_{2}=(S_{2},I_{2})$.
\end{proof}

\begin{theorem}
Let $0<\mathcal{R}_{0}\leq1$. For system (\ref{ruanmod}), if $ \mathcal{R}_{0} ^{*} \leq 1 $ then there are no positive endemic equilibria. Otherwise, if
$ \mathcal{R}_{0} ^{*}>1 $ the following propositions hold:

\begin{enumerate}
\item If $\mathcal{R}_{0}=1$ and $(\beta_{2},\alpha_{2})\in A_{3}$ the system
(\ref{ruanmod}) has a unique positive endemic equilibrium $E_2=(S_2
,I_2)$, where
\[
I_2=-B/A,\quad S_2=\dfrac{1+\alpha I_2}{\beta}\left(  p\delta
+\gamma+\dfrac{\beta_{2}}{1+\alpha_{2}I_2}\right)  .
\]
.

\item If $\max\{P_{1},R_{0}^{+}\}<\mathcal{R}_{0}<1$ and $(\beta_{2}%
,\alpha_{2})\in A_{3},$ the system (\ref{ruanmod}) has a pair of positive
endemic equilibria $E_{1},E_{2}$.

\item If $1>\mathcal{R}_{0}=R_{0}^{+}>P_{1}$ and $(\beta_{2},\alpha_{2})\in
A_{3},$ the system (\ref{ruanmod}) has a unique positive endemic equilibrium
$E_{1}=E_{2}$.

\item If $1>\mathcal{R}_{0}=P_{1}$ and $(\beta_{2},\alpha_{2})\in A_{3},$ the
system (\ref{ruanmod}) has no positive endemic equilibria.

\item If $0<\mathcal{R}_{0}\leq1$ and $(\beta_{2},\alpha_{2})\in A_{1}\cup
A_{2},$ the system (\ref{ruanmod}) has no positive endemic equilibria.

\item If $(\beta_{2},\alpha_{2})\in A_{3}$ and $0<\mathcal{R}_{0}<\max
(R_{0}^{+},P_{1})<1,$ then there are no positive endemic equilibria.
\end{enumerate}

\label{teo4}
\end{theorem}

\begin{proof}

If $0<\mathcal{R}_{0}\leq1,$ then $C\geq0$, so the roots of the
equation $AI^{2}+BI+C=0$ are not real with different sign, but real with equal
signs, complex conjugate or some of them are zero. If endemic equilibria exist
and are positive, it is necessary that $B<0$. After some calculations we can see that:
\begin{align}
&  B<0 \Leftrightarrow\mathcal{R}_{0}>1+\dfrac{\beta(\gamma+\beta_{2}%
+bm-bm\alpha_{2})+\beta mb\alpha+b\alpha_{2}(p\delta+\gamma)}{b\alpha
(p\delta+\gamma+\beta_{2})}:=P_{1}.
\end{align}
From the assumption that $\mathcal{R}_{0}\leq1$ then $P_{1}<1$, hence the
expression $\beta(\gamma+\beta_{2}+bm-bm\alpha_{2})+\beta mb\alpha+b\alpha
_{2}(p\delta+\gamma)$ must be negative, this happens if and only if %
\begin{gather}
\beta_{2}<-\frac{1}{\beta}(b\alpha_{2}(p\delta+\gamma-\beta
m)+\beta(\gamma+bm+mb\alpha))=g(\alpha_{2}).
\end{gather}
If $ \mathcal{R}_{0} ^{*} \leq 1 $ then $-\frac{1}{\beta}(b\alpha_{2}%
(p\delta+\gamma-\beta m)+\beta(\gamma+bm+mb\alpha))<0$ and it is not possible
to find a value of $\beta_{2}$ fulfilling the previous inequality, therefore
there are no  positive endemic equilibria. \par
Now, if $ \mathcal{R}_{0} ^{*}>1 $ we have that.

\begin{enumerate}
\item If $\mathcal{R}_{0}=1$ then $C=0$ and the equation (\ref{eqec7}) is
transformed into
\begin{equation}
AI_{0}^{2}+BI_{0}=0,
\end{equation}

with $A>0$. Its roots are $I_{1}=0$ and $I_2=-B/A,$ and there
exists a unique endemic equilibrium that is positive if and only if $B<0$,  that is given by $E_2=(S_2,I_2)$, where
\begin{align}
I_2 &  =-B/A\nonumber\\
S_2 &  =\dfrac{1+\alpha I_2}{\beta}\left(  p\delta+\gamma+\dfrac{\beta
_{2}}{1+\alpha_{2}I_2}\right)  . \label{eqec13}%
\end{align}
Note that if $\alpha_{2}> \alpha_2^{0} \quad
\text{and} \quad \mathcal{R}_{0} ^{*}>1 $ then $g(\alpha_{2})>0$.

Hence  $A_{3}$ is nonempty and its elements satisfy $B<0$, therefore if
$(\beta_{2},\alpha_{2})\in A_{3}$ there exists a unique positive endemic
equilibrium $E_2$.

\item If $\mathcal{R}_{0}<1$ then $C>0$ and the roots of the quadratic equation
for $I_{0}$  must be real of
equal sign or complex conjugate. By the previous part we know that if
$(\beta_{2},\alpha_{2})\in A_{3}$ then $P_{1}<1$, moreover if $\mathcal{R}%
_{0}>P_{1}$ then $B<0$ and therefore both roots must have positive real part.
Finally, to assure that equilibria are both real, we demand that $\Delta\geq0$
. Computing $\Delta$:%
\begin{align}
\Delta &  =B^{2}-4AC\nonumber\\
&  =A_{2}\mathcal{R}_{0}^{2}+B_{2}\mathcal{R}_{0}+C_{2}=\Delta(\mathcal{R}%
_{0}),
\end{align}
where:
\begin{align}
A_2 &= {\alpha}^{2}{b}^{2}\left(  p\delta+\gamma+\mathit{\beta_{2}}\right)
^{2}\\
B_2 &= -2\,[\beta\,\left(  \gamma+\mathit{\beta_{2}}+bm\left(  1-\mathit{\alpha
_{2}}\right)  \right)  +\alpha\,b\left(  p\delta+\gamma+\mathit{\beta_{2}%
}\right)  +\beta\,mb\alpha\nonumber\\
&  +b\mathit{\alpha_{2}}\,\left(  p\delta+\gamma\right)  ]\alpha\,b\left(
p\delta+\gamma+\mathit{\beta_{2}}\right)  +4\,\mathit{\alpha_{2}}\,\left(
\beta\,\left(  \gamma+bm\right)  \alpha\,b\left(  p\delta+\gamma\right)
\right) \nonumber \\
& b\left(  p\delta+\gamma+\mathit{\beta_{2}}\right) \\
C_2 &= \left(  \beta\,\left(  \gamma+\mathit{\beta_{2}}+bm\left(
1-\mathit{\alpha_{2}}\right)  \right)  +\alpha\,b\left(  p\delta
+\gamma+\mathit{\beta_{2}}\right)  +\beta\,mb\alpha+b\mathit{\alpha_{2}%
}\,\left(  p\delta+\gamma\right)  \right)  ^{2}\nonumber\\
&  -4\,\mathit{\alpha_{2}}\,\left(  \beta\,\left(  \gamma+bm\right)
+\alpha\,b\left(  p\delta+\gamma\right)  \right)  b\left(  p\delta
+\gamma+\mathit{\beta_{2}}\right).
\end{align}
The previous expression is a quadratic function of $\mathcal{R}_{0}$. To
establish the region where $\Delta\geq0$, it is necessary to know how the roots
of $\Delta(\mathcal{R}_{0})$ behave. The discriminant of the quadratic
function $\Delta(\mathcal{R}_{0})$ is
\begin{align}
\Delta_{2}  &  =-16\,\mathit{\alpha_{2}}\,{b}^{2}\beta\,\left(  p\delta
+\gamma+\mathit{\beta_{2}}\right)  ^{2}\left(  \beta\,\gamma+\beta
\,bm+\alpha\,bp\delta+\alpha\,b\gamma\right) \nonumber\\
&  \left(  \alpha(\alpha bm+\beta_{2}+\gamma+bm)-\alpha_{2}(\gamma+bm+\alpha
bm)\right)  . \label{ec15}%
\end{align}
If we assume that $\Delta_{2}<0,$ then $\alpha_{2}<\frac{\alpha(bm\alpha
-\beta_{2}+\gamma+bm)}{\gamma+bm+\alpha bm}$ and in this case we have that:%
\begin{align}
\gamma+\beta_{2}+bm-bm\alpha_{2}+bm\alpha>\frac{2\beta_{2}\alpha
bm+(\gamma+bm)(\gamma+\beta_{2}+bm+bm\alpha)}{\gamma+bm+\alpha bm}>0.
\end{align}
So we get that $P_{1}>1>\mathcal{R}_{0}$, which is a contradiction with
the assumption in this part, therefore $\Delta_{2}\geq0$ and in consequence
$\Delta(\mathcal{R}_{0})$ has two real roots,
\begin{align}
R_{0}^{-}  &  =\dfrac{-B_{2}-\sqrt{\Delta_{2}}}{2A_{2}}\nonumber\\
&  =1-\dfrac{1}{b\alpha^{2}(p\delta+\gamma+\beta_{2})}[\sqrt{-\beta
(\alpha(bm\alpha+\beta_{2}+\gamma+bm-bm\alpha_{2})-\alpha_{2}(\gamma
+bm))}\nonumber\\
&  +\sqrt{\alpha_{2}(\beta(\gamma+bm)+\alpha b(p\delta+\gamma))}%
]^{2},\nonumber\\
R_{0}^{+}  &  =\dfrac{-B_{2}+\sqrt{\Delta_{2}}}{2A_{2}}\nonumber\\
&  1-\dfrac{1}{b\alpha^{2}(p\delta+\gamma+\beta_{2})}[\sqrt{-\beta
(\alpha(bm\alpha+\beta_{2}+\gamma+bm-bm\alpha_{2})-\alpha_{2}(\gamma
+bm))}\nonumber\\
&  -\sqrt{\alpha_{2}(\beta(\gamma+bm)+\alpha b(p\delta+\gamma))}]^{2}.%
\end{align}
Note that due to the positivity of $\Delta_{2}$ and (\ref{ec15}),
we have that
 $$-\beta(\alpha(bm\alpha+\beta_{2}+\gamma+bm-bm\alpha
_{2})-\alpha_{2}(\gamma+bm))$$
 is positive, allowing its roots to be well
defined. Analyzing the derivative of
$\Delta(\mathcal{R}_{0})$ we have that $$\Delta^{\prime}(R_{0}^{+}%
)=\sqrt{\Delta_{2}}>0 \quad \text{and} \quad \Delta^{\prime}(R_{0}^{-})=-\sqrt{\Delta_{2}}<0,$$

moreover $R_{0}^{-}<R_{0}^{+}$ making $\Delta$ positive for $\mathcal{R}%
_{0}>R_{0}^{+}$ or $\mathcal{R}_{0}<R_{0}^{-}$. Nevertheless $$R_{0}%
^{-}=1+\frac{1}{b\alpha(p\delta+\gamma+\beta_{2})}(\beta(\gamma+\beta
_{2}+bm-bm\alpha_{2}+bm \alpha))-\epsilon,$$ while $$P_{1}=1+\frac
{1}{b\alpha(p\delta+\gamma+\beta_{2})}(\beta(\gamma+\beta_{2}+bm-bm\alpha
_{2}+bm \alpha))+\epsilon_{2},$$
with $ \epsilon, \epsilon_{2}>0$, making $R_{0}^{-}<P_{1}<R_{0}$.
Therefore for $\mathcal{R}_{0}>\max(P_{1},R_{0}^{+})$, we have that there exists
 two positive endemic equilibria
$E_{1},E_{2}$, proving this part.

\item If $(\beta_{2},\alpha_{2})\in A_{3}$ then $P_{1}<1$. If $1>\mathcal{R}%
_{0}>P_{1},$ then we have that $B<0$ and $C>0$, therefore we have a pair of
roots of the quadratic for $I$ with positive real part. In the previous part
it was proven that for $P_{1}<1$ the discriminant $\Delta_{2}\geq0$ and both
roots $R_{0}^{+},R_{0}^{-}$ are real and less than one. If $\mathcal{R}%
_{0}=R_{0}^{+}$ then $\Delta=0$ and both roots are fused in one $I_{1}%
=-B/2A=I_{2}$. Therefore we have a unique positive endemic equilibrium
$E_{1}=E_{2}$.

\item If $(\beta_{2},\alpha_{2})\in A_{3}$ then $P_{1}<1$. If $\mathcal{R}%
_{0}=P_{1}<1$ then $C>0,$ implying that the roots are complex conjugate or
real of the same sign. Being $\mathcal{R}_{0}=P_{1}$ then $B=0,$ implying that
both roots have real part equal to zero, therefore there are no positive
endemic equilibria.

\item If $0<\mathcal{R}_{0}\leq1$ and $(\beta_{2},\alpha_{2})\in A_{1}\cup
A_{2}$ then $P_{1}\geq1$, therefore $\mathcal{R}_{0}\leq P_{1}$ , $B\geq0,$
and $C\geq0$. Hence there are two roots with real part zero or negative, which
are not positive equilibria.

\item If $(\beta_{2},\alpha_{2})\in A_{3}$ we have that $P_{1}<1$ and the
roots of the discriminant $R_{0}^{+},R_{0}^{-}$ are real, in addition that
$R_{0}^{-}<P_{1}$ and $R_{0}^{+}<1$ by definition of this case. If
$0<\mathcal{R}_{0}<\max\{R_{0}^{+},P_{1}\}<1,$ then $C>0$ and the roots
$I_{2},I_{3}$ are complex conjugate or real with the same sign. If
$\mathcal{R}_{0}<P_{1}$ then $B>0,$ and the roots have negative real part, so
there are not positive endemic equilibria. If $0<\mathcal{R}_{0}<R_{0}^{+}$
and $\mathcal{R}_{0}>R_{0}^{-}$, then $\Delta<0$ and the roots are complex
conjugate, therefore there is not real endemic equilibria. If $0<\mathcal{R}
_{0}<R_{0}^{+}$ and $\mathcal{R}_{0}\leq R_{0}^{-}<P_{1},$ then it reduces to
the first case in which there are not positive endemic equilibria.
\end{enumerate}

\begin{figure}[t]
\centering
\includegraphics[scale=0.6]{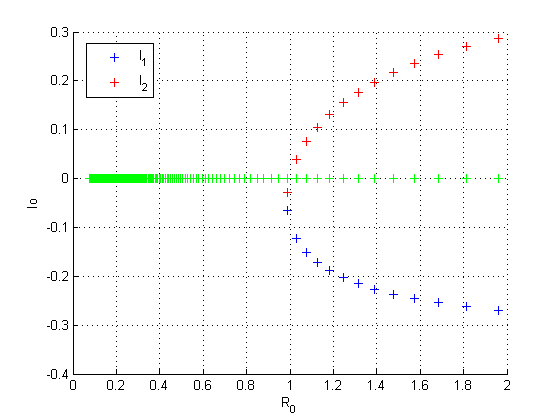}\caption{Graph of
$\mathcal{R}_{0}$ versus the values $I$ of equilibria.  Parameter values used
are $\alpha=0.4,\alpha_{2}=3.8,\beta=0.2,b=0.2,\gamma=0.01,\delta
=0.01,p=0.02,m=0.1$. In this example
$\beta_{2}$ varies from $0$ to $0.025$, therefore $\mathcal{R}_{0}$ varies
between $0.5682$ and $1.9682$. $g(\alpha_2)=   -0.0017$ and $ \alpha_2^{0} =  3.8776 $, so $ ( \beta_2, \alpha_2 ) \in A_ 1 \cup A_2$.  \textit{Forward bifurcation} can be observed
in $\mathcal{R}_{0}=1$.}
\label{fig12}%
\end{figure}

\begin{figure}[t]
\centering
\includegraphics[scale=0.6]{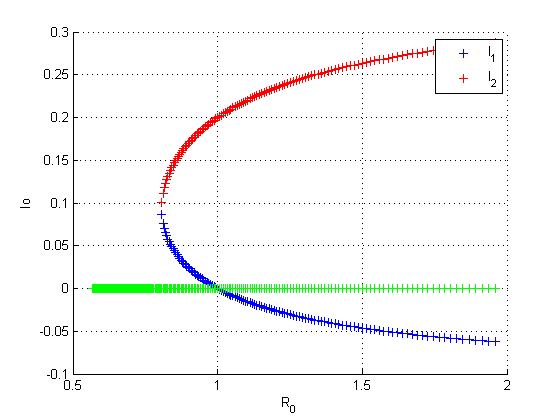}\caption{Graph of
$\mathcal{R}_{0}$ versus the values $I$ of equilibria. In this example
$\beta_{2}$ varies from $0$ to $0.025$ therefore $\mathcal{R}_{0}$ varies
between $0.5682$ and $1.9682$. Parameter values used
are $\alpha=0.4,\alpha_{2}=16,\beta=0.2,b=0.2,\gamma=0.01,\delta
=0.01,p=0.02,m=0.1$. $g(\alpha_2)=   0.1188$ and $ \alpha_2^{0} =  3.8776 $, so $ ( \beta_2, \alpha_2 ) \in A_3 $ . \textit{Backward bifurcation} can be observed
in $\mathcal{R}_{0}=1$ and the existence of two positive endemic equilibria
whenever $\max\{P_{1},R_{0}^{+}\}<\mathcal{R}_{0}<1$.} %
\label{fig2}%
\end{figure}

Theorem \ref{teo4} gives us a complete  scenario of the existence of endemic equilibria. When $\mathcal{R}_{0} ^{*} \leq  1 $ we have that $\mathcal{R}_{0} <1 $, it follows from the fact that $\mathcal{R}_{0} < \mathcal{R}_{0} ^{*}$ whenever $ \beta_2>0 $; then system \ref{ruanmod} has only a disease free equilibrium and no endemic equilibria. \par
Otherwise, when $\mathcal{R}_{0} ^{*}>1$ . If $ ( \beta_2, \alpha_2 ) \in A_1 \cup A_2 $ then we have no endemic equilibria for $0<\mathcal{R}_{0}<1 $ and  a unique endemic equilibria $E_2$ when $\mathcal{R}_{0} >1$, so there exists a forward bifurcation in $\mathcal{R}_{0} =1$ from the disease free equilibrium to $E_2$ (see figure \ref{fig12} ). If $ ( \beta_2, \alpha_2 ) \in A_3$ there exist two positive endemic equilibria whenever $\max\{P_{1},R_{0}^{+}\}<\mathcal{R}_{0}<1$ ( $P_1$ and $\mathcal{R}_{0} ^{+}$ depend on $ \beta_2 $), we can observe the backward bifurcation of the equilibrium $E$ to two endemic equilibria (see figure \ref{fig2} ).

As an immediate consequence of the previous theorem we have that if
$\mathcal{R}_{0}>1$ there exists a unique positive endemic equilibrium, while
if $\mathcal{R}_{0}<1$ and the conditions of the second part are fulfilled,
there exist two positive endemic equilibria. Hence we have the following corollary:
\end{proof}

\begin{corollary}
If $\mathcal{R}_{0}=1$, $ \mathcal{R}_{0} ^{*}>1 $ and $(\beta,\alpha_{2})\in A_{3}$, system
(\ref{ruanmod}) has a backward bifurcation of the disease-free equilibrium $E$.
\end{corollary}

\begin{proof}
First we note that if $(\beta_{2},\alpha_{2})\in A_{3}$ then $R_{0}^{+}$ is
real less than one and $P_{1}<1$, therefore we can find a neighborhood of
points in the interval $(\max\{R_{0}^{+},P_{1}\},1)$. By  case  2 of
theorem, if $\mathcal{R}_{0}$ lies in this neighborhood there exist two
positive endemic equilibria $E_{1},E_{2}$; for $\mathcal{R}_{0}=1$ there
exists a unique positive endemic equilibrium $E_2$, while the other endemic
equilibrium becomes zero. Finally for $\mathcal{R}_{0}>1$ there exists a
unique positive endemic equilibrium as the zero "endemic" equilibrium becomes
negative.
\end{proof}

\section{Characteristic Equation and Stability}

The characteristic equation of the linearization of system (\ref{ruanmod}) in
the equilibrium $(S_{0},I_{0})$ is given by:
\begin{equation}
\det(DF-\lambda I),
\end{equation}
where
\begin{equation}
DF=\left(
\begin{matrix}
\frac{\partial f_{1}}{\partial S} & \frac{\partial f_{1}}{\partial I}\\
\frac{\partial f_{2}}{\partial S} & \frac{\partial f_{2}}{\partial I}%
\end{matrix}
\right)  .
\end{equation}
Matrix is evaluated in the equilibrium $(S_{0},I_{0})$. Functions
$f_{1},f_{2}$ are the following:
\begin{align}
f_{1}  &  =-\dfrac{\beta SI}{1+\alpha I}-bS+bm(1-I)+p\delta I\\
f_{2}  &  =\dfrac{\beta SI}{1+\alpha I}-p\delta I-\gamma I-\dfrac{\beta_{2}%
I}{1+\alpha_{2}I}.
\end{align}

Computing the matrix $DF$ we obtain:%

\begin{equation}
DF (S,I) = \left(
\begin{matrix}
\dfrac{-\beta I}{1 + \alpha I } - b & \dfrac{-\beta S }{ (1+ \alpha I)^{2} }
-bm + p \delta\\
\dfrac{\beta I}{1 + \alpha I } & \dfrac{\beta S}{ (1+ \alpha I)^{2} } - p
\delta- \gamma- \dfrac{ \beta_{2} }{(1+ \alpha_{2} I)^{2}}%
\end{matrix}
\right).
\end{equation}

\subsection{Stability of \textit{disease free }equilibrium}

For the disease free equilibrium $E=(m,0)$ the Jacobian matrix is:
\[
DF(m,0)=\left(
\begin{matrix}
-b & -\beta m-bm+p\delta\\
0 & \beta m-p\delta-\gamma-\beta_{2}%
\end{matrix}
\right).
\]

\begin{theorem}
If $\mathcal{R}_{0}<1$ then the equilibrium $E=(m,0)$ of model (\ref{ruanmod})
is locally asymptotically stable, while if $\mathcal{R}_{0}>1$ then it is
unstable. \label{teo5}
\end{theorem}

\begin{proof}
The characteristic equation for the equilibrium $E$ is given by
\begin{align}
P(\lambda)  &  =\det(DF(m,0)-\lambda I_{2x2})\nonumber\\
&  =\det\left(
\begin{matrix}
-b-\lambda & -\beta m-bm+p\delta\\
0 & \beta m-p\delta-\gamma-\beta_{2}-\lambda
\end{matrix}
\right) \nonumber\\
&  =(-b-\lambda)(\beta m-p\delta-\gamma-\beta_{2}-\lambda). \label{carec1}%
\end{align}
The equation (\ref{carec1}) has two real roots $\lambda_{1}=-b$ and
$\lambda_{2}=\beta m-p\delta-\gamma-\beta_{2}$. By Hartman-Grobman's theorem,
if the roots of (\ref{carec1}) have non-zero real part then the solutions of
system (\ref{ruanmod}) and its linearization are qualitatively equivalent. If
both roots have negative real part then the equilibrium $E$ is locally
asymptotically stable, whilst if any of the roots has positive real part the
equilibrium is unstable. Clearly $\lambda_{1}<0,$ but $\lambda_{2}<0$ if and
only if
\[
\beta m-p\delta-\gamma<\beta_{2},
\]
if and only if $\mathcal{R}_{0}<1$.
\end{proof}

According to the previous theorem and theorem \ref{teo4} we obtain the
following result for the global stability of equilibrium $E$\ :

\begin{theorem}
If $0<\mathcal{R}_{0}<1$ and one of the following conditions holds:

\begin{itemize}
\item $ \mathcal{R}_{0} ^{*} \leq 1 $.

\item $\mathcal{R}_{0}=P_{1}$ and $(\beta_{2},\alpha_{2})\in A_{3}$.

\item $(\beta_{2}, \alpha_{2}) \in A_{1} \cup A_{2} $.

\item $(\beta_{2},\alpha_{2})\in A_{3}$ and $0<\mathcal{R}_{0}<\max\{R_{0}%
^{+},P_{1}\}$.
\end{itemize}

Then equilibrium $E$ of system (\ref{ruanmod}) is globally asymptoticaly stable.
\label{teo6}
\end{theorem}

\begin{proof}
If $0<\mathcal{R}_{0}<1$ then by theorem \ref{teo5} the equilibrium $E$ is
locally asymptotically stable. If any of the given conditions holds then by
theorem \ref{teo4} there are no endemic equilibria in the region
$D=\{S(t),I(t)\geq0\quad\forall t>0,\quad S(t)+I(t)\leq1\}$, which it was
proven to be positively invariant in theorem \ref{teo1}. By \cite{perko} (page
245) any solution of (\ref{ruanmod}) starting in $D$ must approach either an
equilibrium or a closed orbit in $D$. By \cite{kelley} (theorem 3.41) if the
solution path approaches a closed orbit, then this closed orbit must enclose
an equilibrium. Nevertheless, the only equilibrium existing in $D$ is $E$ and
it is located in the boundary of $D$, therefore there is no closed orbit
enclosing it, totally contained in $D$. Hence any solution of system
(\ref{ruanmod}) with initial conditions in $D$ must approach the point $E$ as
$t$ tends to infinity. \begin{figure}[t]
\centering
\includegraphics[scale=0.5]{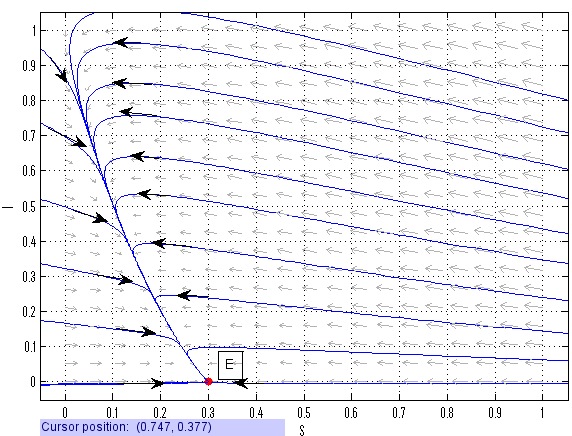}\caption{Global stability of
equilibrium $E$.}%
\label{fig3}%
\end{figure}
\end{proof}

\begin{example}
Take the following values for the parameters: $\alpha=0.4,\alpha_{2}%
=10,\beta=0.2,b=0.2,\gamma=0.01,\delta=0.01,p=0.02,m=0.3,\beta_{2}=0.1$.
Equilibrium $E=(0.3,0)$, $\mathcal{R}_{0}=0.5445<1$. By theorem \ref{teo5}, $E$
is locally asymptotically stable, $ \alpha_2^{0} =7.42<\alpha_{2}$ and $g(\alpha_{2})=-0.1864<\beta_{2}$,
therefore $(\beta_{2},\alpha_{2})\in A_{2}$. By theorem \ref{teo4} there are
no positive endemic equilibria. Finally by theorem \ref{teo6} we have that $E$
is globally stable. See figure \ref{fig3}.
\end{example}

\begin{theorem}
If $\mathcal{R}_{0}=1$ and $\beta_{2}\neq g(\alpha_{2})$ then equilibrium $E$
is a saddle point. Moreover, if $(\beta_{2},\alpha_{2})\in A_{1}\cup A_{2}$
the region $D$ is contained in the stable manifold of $E$. \label{teo7}
\end{theorem}

\begin{proof}
If $\mathcal{R}_{0}=1$ one of the eigenvalues of the Jacobian matrix of the
system is zero, hence we cannot apply Hartman-Grobman's theorem. In order to
establish the stability of equilibrium $E$ we apply central manifold theory.
Making the change of variables, $\hat{S}=S-m$, $ \hat{I}=I $,  we obtain the equivalent system%
\begin{align}
\dfrac{d\hat{S}}{dt}  &  =-\dfrac{\beta(\hat{S}+m) \hat{I} }{1+\alpha \hat{I} }-b\hat
{S}-bm \hat{I} +p\delta \hat{I} \nonumber\\
\dfrac{d \hat{I} }{dt}  &  =\dfrac{\beta(\hat{S}+m) \hat{I} }{1+\alpha \hat{I} }-p\delta \hat{I} -\gamma
\hat{I} -\dfrac{\beta_{2} \hat{I} }{1+\alpha_{2} \hat{I} }. \label{carac6}%
\end{align}
Because $ \hat{I}=I $ we ignore the hat and use only $I$. This new system has an equilibrium in $\hat{E}=(0,0)$ and its Jacobian matrix
in that point is%
\begin{equation}
DF(m,0)=\left(
\begin{matrix}
-b & - \beta m - bm + p \delta\\
0 & 0
\end{matrix}
\right)  .\label{eqst1}%
\end{equation}

Using change of variables $S=u-{\frac{\left(  \gamma+\mathit{beta2}+bm\right)  v}{b}},I=v$
and $\beta m=p\delta+\gamma+\beta_{2}$ we obtain the equivalent system (see appendix A):%
\begin{align}
\frac{dv}{dt}  &  =0u+f(v,u)\nonumber\\
\frac{du}{dt}  &  =-bu+g(v,u), \label{eqst2}%
\end{align}
where $f$ and $g$ are defined in Appendix A. \par
By \cite{carr}, system \eqref{ruanmod} has a center manifold of the form
$u=h(v)$ and the flow in the
center manifold ( and therefore in the system ) is given by the equation
\[
v^{\prime}=f(v,h(v))\sim f(v,\phi(v)),
\]
where $h(v) = a_0 v^{2} + a_1 v^{3} + O(v^{4}) $, and $a_i$'s are given in Appendix A. Expanding the Taylor series of $f$ we obtain the flow equation%
\begin{align}
v^{\prime}  &  =-{\frac{{b}^{3}\beta\,m+{b}^{2}{\beta}^{2}m+{b}^{3}%
\gamma\,\mathit{\alpha_{2}}-{b}^{2}\beta\,p\delta+{b}^{3}\alpha\,\beta
\,m+{b}^{3}p\delta\,\mathit{\alpha_{2}}-{b}^{3}\beta\,m\mathit{\alpha_{2}}%
}{{b}^{3}}}v^{2}+O(v^{3})\nonumber\\
&  =Hv^{2}+O(v^{3}).
\end{align}
Therefore the dynamics of solutions near the equilibrium $\hat{E}=(0,0)$
is given by the quadratic term, whenever this term is not zero. We note that
$H=0$ if and only if
\begin{equation}
\alpha_{2}=\frac{-\beta(bm+\beta m-p\delta+b\alpha m)}{b(p\delta+\gamma-\beta
m)}.
\end{equation}
Substituting again $\mathcal{R}_{0}=1$, expressed as $\beta m=p\delta
+\gamma+\beta_{2}$, we obtain $H=0$ if and only if $\beta_{2}=g(\alpha_{2})$. \par

If $(\beta_{2},\alpha_{2})\in A_{3}$ then $H>0$. $v^{\prime}>0$ for $v\neq0$.
If $(\beta_{2},\alpha_{2})\in A_{1}\cup A_{2}$ then $H<0$, $v^{\prime}<0$ for
$v\neq0$. In both cases $\hat{E}$ is a saddle point. Moreover, if $(\beta
_{2},\alpha_{2})\in A_{1}\cup A_{2}$ then $H<0$ and $v^{\prime}<0$ for $v>0$.
Recalling $v(t)=I(t)$ we have under this assumption that $I^{\prime}(t)<0$ for
$I>0$ therefore $I(t)\rightarrow0^{+}$, while as $v_{1}=(1,0)$ is the stable
direction of the point $E$ then $S(t)\rightarrow0$, therefore the solutions in
the region $D$ approach the equilibrium $E$ as $t\rightarrow\infty$.
\begin{figure}[t]
\centering
\includegraphics[scale=0.5]{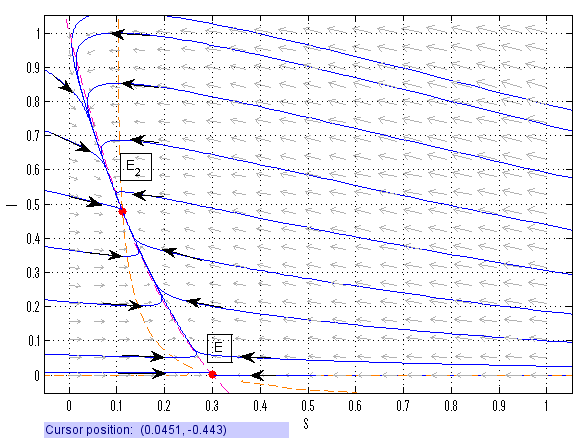}\caption{Phase plane of the system
for $\mathcal{R}_{0}=1$ with $(\beta_{2},\alpha_{2})\in A_{3}$}%
\label{fig9}%
\end{figure}
\end{proof}

\begin{example}
Take the following values for the parameters: $\beta=0.2,\alpha=0.4,\delta
=0.01,\gamma=0.01,\alpha_{2}=10,m=0.3,p=0.02,b=0.2, \beta_2 = 0.0498$. In this case
$\mathcal{R}_{0}=1$, $ \alpha_2^{0} =1.8876$
and  $g(\alpha_{2})=0.4040$, hence $(\beta_{2},\alpha_{2})\in A_{3}$.
By the first case of theorem \ref{teo4} the system has a unique endemic
equilibrium in $S_2=0.11210,I_2=0.4781$. By theorem \ref{teo7} the
equilibrium $E$ is a saddle point, see figure \ref{fig9}.
\end{example}

\begin{figure}[t]
\centering
\includegraphics[scale=0.5]{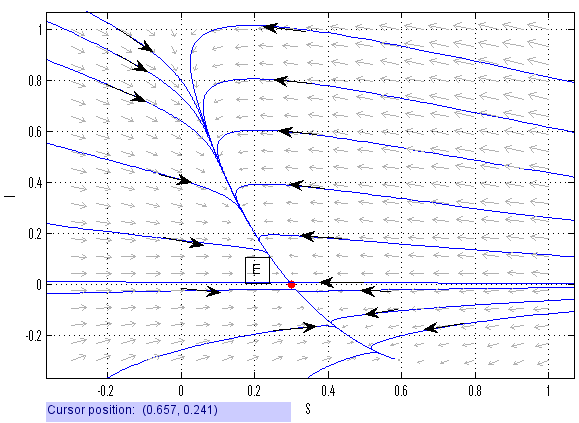}\caption{Phase plane for
$\mathcal{R}_{0}=1$ with $(\beta_{2},\alpha_{2})\in A_{2}$ }%
\label{fig10}%
\end{figure}

\begin{example}
If we take the same values as in the previous example except $\alpha
_{2}=2$, then
$g(\alpha_{2})=0.0056<\beta_{2}$, hence $(\beta_{2},\alpha_{2})\in A_{2}$. By
theorem \ref{teo4} the system has no endemic equilibria, and by theorem
\ref{teo7} the point $E$ is a saddle point. Moreover, the region $D$ is
totally contained in the stable manifold, see figure \ref{fig10}.
\end{example}

\subsection{Stability of endemic equilibria}

The general form of the Jacobian matrix is
\begin{equation}
DF=\left(
\begin{matrix}
-\dfrac{\beta I}{1+\alpha I}-b &  & -\dfrac{\beta S}{(1+\alpha I)^{2}%
}-bm+p\delta\\
\dfrac{\beta I}{1+\alpha I} &  & \dfrac{\beta S}{(1+\alpha I)^{2}}%
-p\delta-\gamma-\dfrac{\beta_{2}}{(1+\alpha_{2}I)^{2}}%
\end{matrix}
\right)  .
\end{equation}
Therefore the characteristic equation for an endemic equilibrium is
\begin{align}
P(\lambda)  &  =\left(  -\dfrac{\beta I}{1+\alpha I}-b-\lambda\right)  \left(
\dfrac{\beta S}{(1+\alpha I)^{2}}-p\delta-\gamma-\dfrac{\beta_{2}}%
{(1+\alpha_{2}I)^{2}}-\lambda\right) \nonumber\\
-  &  \left(  \dfrac{\beta I}{1+\alpha I}\right)  \left(  -\dfrac{\beta
S}{(1+\alpha I)^{2}}-bm+p\delta\right).
\end{align}
If we denote by
\begin{align}
C_{I}  &  :=\frac{\beta I}{1+\alpha I}\\
C_{S}  &  :=\frac{\beta S}{(1+\alpha I)^{2}}\\
D_{I}  &  :=\frac{\beta_{2}}{(1+\alpha_{2}I)^{2}}.%
\end{align}
Then the characteristic polynomial is rewritten as
\begin{align}
P(\lambda)  &  =\lambda^{2}+W\lambda+U \label{carec2}.%
\end{align}
Where:
\begin{align}
W &= C_I+b-C_S + p \delta + \gamma + D_I \\
U &= C_I \gamma + C_I D_I - b C_S + b p \delta + b \gamma + b D_I + C_I b m .
\end{align}

By proposition Routh Hurwitz criteria for $n=2$ if the coefficient  $W$ and the independent
term $U$ are positive then the roots of the characteristic equation have negative
real part and therefore the endemic equilibrium is locally asymptotically stable.
Note that whenever the equilibriums are positive, $C_{I},C_{S},D_{I}$ will be
positive as well. Let us analyze the stability according to the value of
$\mathcal{R}_{0}$.

\begin{theorem}
Whenever the equilibrium $E_1$ exists it is a saddle and therefore unstable.

\label{teo8}
\end{theorem}

\begin{proof}
Consider $E_{1}=(S_{1},I_{1})$ and its characteristic polynomial
(\ref{carec2}). By Routh-Hurwitz criterion for quadratic polynomials, its
roots have negative real part if and only if $U>0$ and $W>0$ , where $U,W$
depend on $E_{1}$. Moreover, when $U<0$ its roots are both real with different sign and when $U>0$ and $W<0$ the roots have positive real part. Computing the value of $U$ and expressing $S_{1}$ in terms
of $I_{1}$ we obtain
\begin{equation}
U=\dfrac{I_1(a_{1}I_{1}^{2}+b_{1}I_{1}+c_{1})}{(1+\alpha I_{1})(1+\alpha
_{2}I_{1})^{2}}=\dfrac{I_1 F(I_{1})}{(1+\alpha I_{1})(1+\alpha_{2}I_{1})^{2}}.
\end{equation}
Where:
\begin{align}
a_{1} &  =\alpha_{2}^{2}(\beta\gamma+bp\alpha\delta+b\alpha\gamma
+bm\beta) = \alpha_ 2 A >0, \nonumber\\
b_{1} &  =2\alpha_{2}(\beta\gamma+bp\alpha\delta+b\alpha\gamma+bm\beta
) = 2A>0,\nonumber\\
c_{1} &  =\beta\beta_{2}+bm\beta+bp\alpha\delta+b\alpha\beta_{2}+\beta
\gamma-b\alpha_{2}\beta_{2}+b\alpha\gamma = B- \alpha_2 C.
\end{align}
We are assuming that equilibrium $E_1$ exists and it is positive, and these happens (by previous section) when $B<0$ and $C>0$, so $c_1<0$. The sign of $U$ is equal to $\sgn(F(I_{1}))$. $F(I_{1})$ has two roots of the
form:
\begin{align}
I\ast  &=\dfrac{-b_{1}+\sqrt{b_{1}^{2}-4a_{1}c_{1}}}{2a_{1}} \\
I\ast\ast &=\dfrac{-b_{1}-\sqrt{b_{1}^{2}-4a_{1}c_{1}}}{2a_{1}}.
\end{align}
Where $b_1^{2}-4a_1 c_1>0$ and therefore $I \ast $ and $I \ast \ast$ are both real values with
$I\ast\ast<0$.  $F(I_{1})>0$ for $I_{1}>I\ast$ and $I_{1}<I\ast\ast$, but second condition never holds because $I_{1}>0$, so $F(I_1)<0$ for $0 < I_1 < I\ast  $. \par

Computing $I*$ in terms of $A,B,C$:
\begin{align}
I \ast &= - \dfrac{1}{\alpha_2} + \dfrac{1}{\alpha_2 A} \sqrt{( A^{2}- \alpha_2 AB + \alpha_2^{2} AC )} .
\end{align}
Substituting $\Delta = B^{^2}-4AC>0$

\begin{align}
I \ast &= - \dfrac{1}{\alpha_2} + \dfrac{1}{\alpha_2 A} \sqrt{\left( A^{2}- \alpha_2 AB + \frac{\alpha_2 ^{2}}{4} (B^{2}- \Delta ) \right)} \nonumber \\
&= - \dfrac{1}{\alpha_2} + \dfrac{1}{2 \alpha_2 A} \sqrt{(2A- \alpha_2 B )^{2}- \alpha_2^{2} \Delta } \nonumber \\
&> - \dfrac{1}{\alpha_2} + \dfrac{1}{2 \alpha_2 A} \left(  \sqrt{(2A- \alpha_2 B )^{2}} - \sqrt{\alpha_2^{2} \Delta}  \right) \nonumber \\
&= \dfrac{-B- \sqrt{\Delta} }{2A} = I_1.
\end{align}
Therefore $U<0$ and the equilibrium $E_1$ is a saddle.
\end{proof}

\begin{theorem}
Assume the conditions of theorem \ref{teo4} for existence and positivity  of the endemic equilibrium $E_2$.
 If $I_2< I*$ the equilibrium $E_2$ is unstable, else if $I_2> I*$  then $E_2$ is locally asymptotically stable for $s>0$ and  unstable for $s<0$. \par
Where $s=m_{1}(-B+\sqrt{B^{2}-4AC})+2Am_{2}$,
\begin{align}
m_{1} &  =(r+\beta_{2}\alpha-\beta_{2}\alpha_{2}+2B\alpha_{2})A^{2}-\alpha
_{2}^{2}rAC-AB\alpha_{2}(b\alpha_{2}+2r+B^{2}\alpha_{2}^{2}r),\nonumber\\
m_{2} &  =bA^{2}-AC\alpha_{2}(b\alpha_{2}+2r)+\alpha_{2}^{2}rBC,\nonumber\\
r &  =\alpha(p\delta+b+\gamma)+\beta.
\end{align}
\label{teo9}
\end{theorem}

\begin{proof}
Consider $E_{2}=(S_{2},I_{2})$ be real and positive,  and its characteristic polynomial (\ref{carec2}). We will have that the equilibrium is unstable when $U<0$ and  locally asymptotically stable when $U>0, W>0$.
 Following the previous proof
\begin{equation}
U=\dfrac{I_2 (a_{1}I_{2}^{2}+b_{1}I_{2}+c_{1})}{(1+\alpha I_{2})(1+\alpha
_{2}I_{2})^{2}}=\dfrac{I_2 F(I_{2})}{(1+\alpha I_{2})(1+\alpha_{2}I_{2})^{2}}.
\end{equation}
Where  $a_1,b_1,c_1$ are the same as in previous theorem.  Therefore $\sgn(U)=\sgn(F(I_{2}))$. We have seen that $F(I_{2})$ has two real roots $I \ast$ and $ I \ast \ast$. Again $F(I_{2})>0$ for $I_{2}>I\ast$ and $I_{2}<I\ast\ast$ (which does not holds because $I \ast \ast<0$), and $F(I_2)<0$ for $0 < I_2 < I\ast  $ . So if $I_2< I*$ the equilibrium $E_2$ is unstable.  \par
When $I_2> I*$ then $U>0$ and

\begin{align}
&  W=\dfrac{1}{(1+\alpha I_{2})(1+\alpha_{2}I_{2})^{2}}[{\alpha_{{2}}}%
^{2}\left(  \alpha\,\gamma+b\alpha+\beta+\alpha\,p\delta\right)  {I_{{2}}}%
^{3}\nonumber\\
&  +\alpha_{{2}}\left(  b\alpha_{{2}}+2\,\alpha\,p\delta+2\,b\alpha
+2\,\alpha\,\gamma+2\,\beta\right)  {I_{{2}}}^{2}\nonumber\\
&  +\left(  \alpha\,p\delta+b\alpha+\beta+\alpha\,\mathit{\beta_{2}%
}-\mathit{\beta_{2}}\,\alpha_{{2}}+\alpha\,\gamma+2\,b\alpha_{{2}}\right)
I_{{2}}+b]\nonumber\\
&  =\dfrac{G(I_{2})}{(1+\alpha I_{2})(1+\alpha_{2}I_{2})^{2}}.
\end{align}
By using the division algorithm,
\begin{align}
G(I_{2}) &  =(AI_{2}^{2}+BI_{2}+C)P(I_{2})\nonumber\\
&  +\frac{1}{A^{2}}[(r+\beta_{2}\alpha-\beta_{2}\alpha_{2}+2B\alpha_{2}%
)A^{2}-\alpha_{2}^{2}rAC-AB\alpha_{2}(b\alpha_{2}+2r+B^{2}\alpha_{2}%
^{2}r)I_{2}\nonumber\\
&  +bA^{2}-AC\alpha_{2}(b\alpha_{2}+2r)+\alpha_{2}^{2}rBC],\nonumber\\
&  =(AI_{2}^{2}+BI_{2}+C)P(I_{2})+\dfrac{m_{1}I_{2}+m_{2}}{A^{2}}.
\end{align}
Where $P(I_{2})$ is a polynomial in $I_{2}$ of degree one. Being $I_{2}$ a
coordinate of an equilibrium then $AI_{2}^{2}+BI_{2}+C=0$ and
\[
G(I_{2})=\dfrac{m_{1}I_{2}+m_{2}}{A^{2}}.
\]
Hence $\sgn(W)=\sgn(G(I_{2}))=\sgn(\frac{m_{1}I_{2}+m_{2}}{A^{2}})=\sgn(m_{1}%
I_{2}+m_{2}).$ Substituting the value of $I_{2},$
\[
m_{1}I_{2}+m_{2}=\frac{m_{1}}{2A}(-B+\sqrt{B^{2}-4AC})+m_{2}.
\]
It follows that $\sgn(m_{1}I_{2}+m_{2})=\sgn(m_{1}(-B+\sqrt{B^{2}-4AC}%
)+2Am_{2})=\sgn(s).$ Therefore $E_2$ is unstable if $s<0$ and locally asymptotically stable if $s>0$.

\end{proof}

\section{Hopf bifurcation}

By previous section we know that the system \eqref{ruanmod} has two positive endemic equilibria under the conditions of theorem  \eqref{teo4} . Equilibrium $E_1$ is always a saddle, so its stability does not change and there is no possibility of a Hopf bifurcation in it. So let us analyse the existence of a Hopf bifurcation of equilibrium $E_2=(S_2,I_2)$. Analysing the characteristic equation for $E_2$,  it has a pair of pure imaginary roots if and only if $U>0$ and $W=0$ . \par

\begin{theorem}
System \eqref{ruanmod} undergoes a Hopf bifurcation of the endemic equilibrium $E_2$ (whenever it exists) if $I_2>I^{*}$ and $s= 0$. Moreover, if $ \bar{a}_{2}<0 $, there is a family of stable periodic orbits of
\eqref{ruanmod} as $s$ decreases from 0; if $ \bar{a}_2>0 $, there is a family of unstable periodic orbits of \eqref{ruanmod} as $s$ increases from 0. \label{biftheo1}
\end{theorem}

The characteristical polinomial for $E_2$ has a pair of pure imaginary roots iff $U>0$ and $W=0$. From the proof of theorem \ref{teo8} we have that $ U>0$ if and only if one of the conditions (i),(ii) is satisfied .

Although, sgn$(W)=$sgn$(s)$, so $W=0$ if and only if $s=0$. By first part of theorem 3.4.2 of \cite{guckenheimer} the roots $ \lambda $ and $ \bar{ \lambda } $ of \eqref{carec2} for $E_2$ vary smoothly, so we can affirm that near $ s=0 $ these roots are still complex conjugate and

\begin{align}
\frac{d  Re( \lambda ( s))}{d s} \mid_{s = 0} &= \frac{d}{d s } \left( \frac{1}{2} W(s) \right)  \nonumber \\
&= \frac{1}{2} \frac{d}{d s } \left( \frac{1}{2 A ^{3} (1+ \alpha_1 I_1)(	1+ \alpha_2 I_1)^{2} } s  \right) \nonumber \\
&=  \frac{1}{ 4 A^{3} (1+ \alpha_1 I_1)(	1+ \alpha_2 I_1)^{2} } \neq 0.
\end{align}

Therefore $s=0$ is the Hopf bifurcation point for \eqref{ruanmod} .\par

To analyze the behaviour of the solutions of \eqref{ruanmod} when $  s=0 $ we make a change of coordinates  to obtain a new equivalent system to \eqref{ruanmod} with an equilibrium in $(0,0)$ in the $x-y$ plane ( see appendix B ). Under this change the system becomes:

\begin{align}
\frac{dx}{dt} &= {\frac {a_{{11}}x+a_{{12}}y+c_{{1}}xy+c_{{2}}{y}^{2} }{1+\alpha_{{1}}y+
\alpha_{{1}}I_{{2}}}}, \nonumber \\
\frac{dy}{dt} &= {\frac {a_{{21}}x+a_{{22}}y+c_{{3}}xy+c_{{4}}x{y}^{2}+c_{{5}}{y}^{2}+c
_{{6}}{y}^{3} }{ \left( 1+\alpha_{{1}}y+\alpha_{{1}}I_{{2}} \right)
 \left( 1+\alpha_{{2}}y+\alpha_{{2}}I_{{2}} \right) }}. \label{hmrhb1}
\end{align}

Where the $a_{ij}$'s and $c_i$'s are defined in appendix B.

System \eqref{hmrhb1} and \eqref{ruanmod} are equivalent ( appendix B ), so we can work with \eqref{hmrhb1}. This system has a pair of pure imaginary eigenvalues if and only if \eqref{ruanmod} has them too. As we said before it happens if and only if any of conditions (i),(ii) is satisfied and $s=0$. Computing jacobian matrix $DF(0,0)$ of \eqref{hmrhb1}
\begin{equation}
DF(0,0)=  \left[ \begin {array}{cc} {\dfrac {a_{{11}}}{1+\alpha_{{1}}I_{{2}}}}&
{\dfrac {a_{{12}}}{ \left( 1+\alpha_{{1}}I
_{{2}} \right) }}\\ \noalign{\medskip}{\dfrac {a_{{21}}}{ \left( 1+
\alpha_{{2}}I_{{2}} \right)  \left( 1+\alpha_{{1}}I_{{2}} \right) }}&{
\dfrac {a_{{22}}}{ \left( 1+\alpha_{{2}}I_{{2}} \right)  \left( 1+
\alpha_{{1}}I_{{2}} \right) }}\end {array} \right].
\end{equation}

 $$ Tr (DF(0,0)) = Tr(Df(S_2,I_2)) , \quad  \det (DF(0,0)) = \det (Df(S_2,I_2)).$$
 So condition $s=0$ is equivalent to $ a_{11}(1+\alpha_2 I_2)+ a_{22}=0 $ and (i),(ii) are equivalent to $ a_{22} a_{11}- a_{12} a_{21} > 0 $.  \par
  System \eqref{hmrhb1} can be rewritten as

\begin{align}
\dfrac{dx}{dt} &= {\dfrac {a_{{11}}x}{1+\alpha_{{1}}I_{{2}}}}+{\frac {a_{{12}}y}{1+\alpha
_{{1}}I_{{2}}}} + G_1 (x,y) \\
\dfrac{dy}{dt} &= {\frac {a_{{21}}x}{ \left( 1+\alpha_{{1}}I_{{2}} \right)  \left( 1+
\alpha_{{2}}I_{{2}} \right) }}+{\frac {a_{{22}}y}{ \left( 1+\alpha_{{1
}}I_{{2}} \right)  \left( 1+\alpha_{{2}}I_{{2}} \right) }} + G_2(x,y).
\end{align}

Where $G_1,G_2$ are defined in appendix B.

 Let $ \Lambda = \sqrt{ \det (DF(0,0)) } $.  We use the change of variable $u=x, v= \frac{a_{11}}{ \Lambda( 1+ \alpha_1  I_2)} + \frac{a_{12} y}{  \Lambda(1 + \alpha_{1} I_2 ) } $, to obtain the following equivalent system:

\begin{equation}
\left( \begin{matrix}
 u \\ v
 \end{matrix}\right)= \left( \begin{matrix}
 0 & \Lambda \\ - \Lambda & 0
 \end{matrix}\right) \left( \begin{matrix}
 u \\ v
 \end{matrix}\right) + \left( \begin{matrix}
 H_1 (u,v) \\ H_2(u,v)
 \end{matrix}\right).
\end{equation}
Where
\begin{align}
H_1(u,v) &= -\dfrac{\left(  \left( -a_{{12}}c_{{1}}+a_{{11}}c_{{2}} \right) u+ \left( -
  \Lambda c_{{2}}\alpha_{{1}}I_{{2}}+  \Lambda a_{{12}
}\alpha_{{1}}-  \Lambda c_{{2}} \right) v \right)  \left(
 \left(  \Lambda+  \Lambda \alpha_{{1}}I_{{2}} \right) v-a_{{11}}u
 \right)
}{a_{{12}} \left(  \left( \alpha_{{1}}  \Lambda +  \Lambda {\alpha_{{1}}}^{2}I_{{2}} \right) v+a_{{12}}-\alpha_{{1}}a_{{
11}}u+a_{{12}}\alpha_{{1}}I_{{2}} \right)
}  \\
H_2(u,v) &= - \dfrac{1}{h(u,v)} \left[ (\Lambda(1+ \alpha_1 I_2)v-a_{11}u) \left( A_1 v^{2}+A_2 uv + A_3 v + A_4 u^{2} + A_5 u \right) \right]
\end{align}
And $A_1,...,A_5, h(u,v)$ are defined in appendix B.
Let
\begin{align}
\bar{a}_2 &= \dfrac{1}{16} [ (H_1)_{uuu} + (H_1)_{uvv} + (H_{2})_{uuv} + (H_2)_{vvv}  ] + \dfrac{1}{16( - \Lambda)}  [(H_1)_{uv} ((H_1)_{uu} + (H_1)_{vv}) \nonumber \\
& - (H_{2})_{uv} ( (H_{2})_{uu} + (H_2)_{vv} ) - (H_1)_{uu}(H_2)_{uu} + (H_1)_{vv} (H_2)_{vv}].
\end{align}
Where $$(H_1)_{uuu}= \dfrac{ \partial }{ \partial u } \left( \dfrac{ \partial }{ \partial u } \left( \dfrac{ \partial H_1}{ \partial u } \right) \right)(0,0) ,$$
and so on ($ \bar{a}_2 $ is explicitly expressed in appendix B) . \par
Then by theorem 3.4.2 of \cite{guckenheimer} if $ \bar{a }_{2} \neq 0  $ then there exist a surface of periodic solutions, if $ \bar{a}_{2}<0 $ then these cycles are stable, but if $ \bar{a}_{2}>0 $ then cycles are repelling. \par

\section{Discussion}

As we said in the introduction, traditional epidemic models have always stability results in terms of $\mathcal{R}_{0} $, such that we need only reduce $\mathcal{R}_{0} <1$ to eradicate the disease. However, including the treatment function brings new epidemic equilibria that make the dynamics of the model more complicated. Now, let's discuss some control strategies for the infectious disease, analysing the parameters of the treatment function ($ \alpha_{2}, \beta_{2} $)  and looking for conditions that allow us to eliminate the disease. We make this study by cases.\par
A first approach is focus on  the definition of $\mathcal{R}_{0} $, we can  see that $\mathcal{R}_{0} $ decreases when $ \beta_2 $ increases, so the first measure suggesting control is a big value for $ \beta_2 $. But this is not always a good way to proceed. Let us divide our analysis in the following cases: \par

\textit{Case 1: There is no positive endemic equilibrium for $\mathcal{R}_{0}  \leq 1$. }  This happens when $\mathcal{R}_{0} ^{*} \leq 1 $ ( by theorem \ref{teo4} )  or when $ \mathcal{R}_{0} ^{*}>1$ and $ ( \alpha_{2}, \beta_{2} ) \in A_1 \cup A_2  $ ( theorem \ref{teo4} , number 5). In this case if  $\mathcal{R}_{0} >1$ there is a unique positive endemic equilibrium, therefore there exists a bifurcation at $\mathcal{R}_{0} =1$  : from the disease free equilibrium, which is globally asymptotically stable for $0<\mathcal{R}_{0} <1$ (by theorem \ref{teo5})  and a saddle for $\mathcal{R}_{0} =1$ and $ \beta_2 \neq g( \alpha_2 ) $ (theorem \ref{teo7} ),  to the positive endemic equilibrium $E_2$ as $\mathcal{R}_{0} $ increase. $E_2$ will be locally asymptotic stable or unstable depending on theorem \ref{teo9}   or surrounded by a limit cycle  (theorem \ref{biftheo1} ) . If  conditions for Hopf bifurcation hold then the stability of the limit cycle is determined by $ \bar{a}_{2} $; when $ \bar{a}_2<0 $  the periodic orbit is stable and therefore $E_2$ is unstable, while if $ \bar{a}_2>0 $ then the periodic orbit is unstable and $E_2$ is stable . In this case the best way to eradicate the disease is finding parameters that allow $\mathcal{R}_{0} <1$, because then all the infectious states tend to $I=0$.   \par

\begin{figure}
\centering
\includegraphics[scale=0.5]{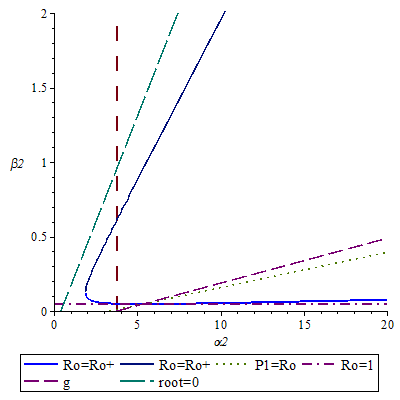}
\caption{Bifurcation diagram in terms of $ \beta_2 $ and $ \alpha_2 $. The values of the parameters taken are $ \alpha = 0.4, \beta = 0.3, b= 0.2, \gamma = 0.03, \delta = 0.05, p=0.3, m=0.3 $.  Here $ \mathcal{R}_{0} < \mathcal{R}_{0} ^{+} $ inside the solid curve ( $\mathcal{R}_{0} =\mathcal{R}_{0} ^{+} $) and $ \mathcal{R}_{0} > \mathcal{R}_{0} ^{+} $ outside it, whenever $( \beta_2, \alpha_2 ) $ is in its domain (under the long dashed line ``root=0'') . $\mathcal{R}_{0} <1$ above the  dot-lined line and $ \mathcal{R}_{0} >1 $ under it; $\mathcal{R}_{0} <P_1$ above the dotted line and $\mathcal{R}_{0} >P_1$ under that one. The areas $A_1, A_2,A_3$ are delimited by  the dashed  line $g=0$  and $ \alpha_2=3.8 $. In this case the endemic equilibria $E_1$ and $E_2$ exists both  in the area  delimited by the line $\mathcal{R}_{0} =1$ and the  dotted line $\mathcal{R}_{0} =P_1$ , while $E_2$ exists by itself under the line $\mathcal{R}_{0} =1$.} \label{biffig1}
\end{figure}

\textit{Case 2: There exist endemic equilibria for $\mathcal{R}_{0} \leq 1$}. This happens when  $ ( \alpha_{2}, \beta_{2} ) \in A_3 $. The existence of endemic equilibria is determined by the relationship between $\mathcal{R}_{0} $ and $ \max \{ P_1, \mathcal{R}_{0} ^{+} \} $. Let $F( \alpha_2, \beta_2 )= \mathcal{R}_{0} -\mathcal{R}_{0} ^{+}$, $G( \alpha_2, \beta_2 ) = \mathcal{R}_{0} - P_1$, and focus on the implicit curves defined by $F=0$ and $G=0$. These curves divide the domain $A_3$ in another ones (see figure \ref{biffig1} ):
\begin{align}
A_{3}^{1} & = \{( \alpha_2, \beta_2 ) \in A_3, 0< \mathcal{R}_{0} <\mathcal{R}_{0} ^{+}  \} \nonumber \\
A_{3}^{2} & = \{( \alpha_2, \beta_2 ) \in A_3,  \mathcal{R}_{0} >\mathcal{R}_{0} ^{+}  \} \nonumber \\
A_{3}^{3} & = \{( \alpha_2, \beta_2 ) \in A_3, 0< \mathcal{R}_{0} <P_1  \} \nonumber \\
A_{3}^{4} & = \{( \alpha_2, \beta_2 ) \in A_3,  \mathcal{R}_{0} > P_1   \}.
\end{align}

 If $ ( \alpha_{2}, \beta_{2} ) \in A_{3}^{2} \cap A_{3}^{4} $ then there exist two endemic equilibria $E_1$( a saddle ) and $E_2$ ( stable or unstable depending on conditions of theorems \ref{teo9}  and possibly with a periodic orbit around (theorems \ref{biftheo1} )), but when $\mathcal{R}_{0} =1$ one of them becomes negative, leaving us with $E_2$. In this case $\mathcal{R}_{0} <1$ is not a sufficient condition to control the disease, because even with $\mathcal{R}_{0} <1$ we have endemic positive equilibria that could be stable  and then the disease will tend to a non zero value; also we have the possibility of a periodic solution, or biologically, an outbreak that will apparently `` disappear '' but will re-emerge after some time. \par
The best way in this case is ensuring $ ( \alpha_{2}, \beta_{2} ) \in ( A_{3}^{2} \cap A_{3}^{4})^{c} $ because then we don't have endemic equilibria for $\mathcal{R}_{0} <1$ and the disease free will be globally asymptotically stable.


\appendix
\section{Computing center manifold}

The Jacobian matrix of system \eqref{carac6} is
\begin{equation}
DF(m,0)=\left(
\begin{matrix}
-b & - \beta m - bm + p \delta\\
0 & 0
\end{matrix}
\right)  . \label{eqst1}%
\end{equation}
 With eigenvalues $\lambda_{1}=-b$ and $\lambda_{2}=0$ and
 respective eigenvectors $v_{1}=(1,0)$ and $v_{2}=(-\frac{\gamma+\beta
_{2}+bm}{b},1)$. Using the eigenvectors to establish a new coordinate system
we define:%
\begin{equation}
\left(
\begin{matrix}
\hat{S}\\
I
\end{matrix}
\right)  =\left(
\begin{array}
[c]{cc}%
1 & -{\frac{\gamma+\beta_{2}+bm}{b}}\\
\noalign{\medskip}0 & 1
\end{array}
\right)  \left(
\begin{matrix}
u\\
v
\end{matrix}
\right)  \quad\text{, or}\quad\left(
\begin{matrix}
u\\
v
\end{matrix}
\right)  =\left(
\begin{array}
[c]{cc}%
1 & {\frac{\gamma+\beta_{2}+bm}{b}}\\
\noalign{\medskip}0 & 1
\end{array}
\right)  \left(
\begin{matrix}
S\\
I
\end{matrix}
\right).
\end{equation}
Under this transformation the system becomes
\begin{align}
\frac{du}{dt}  &  ={\frac{d}{dt}}\hat{S}\left(  t\right)  +{\frac{\left(
\gamma+\mathit{\beta_{2}}+bm\right)  {\frac{d}{dt}}I\left(  t\right)  }{b}%
}\nonumber\\
&  =-{\frac{\beta\,\left(  \hat{S}+m\right)  I}{1+\alpha\,I}}-b\hat
{S}-bmI+p\delta\,I+\left(  \gamma+\mathit{\beta_{2}}+bm\right) \nonumber\\
&  \left(  {\frac{\beta\,\left(  \hat{S}+m\right)  I}{1+\alpha\,I}}-\left(
p\delta+\gamma\right)  I-{\frac{\mathit{\beta_{2}}\,I}{1+\mathit{\alpha_{2}%
}\,I}}\right)  \frac{1}{b},\nonumber\\
\frac{dv}{dt}  &  =\frac{dI}{dt}\\
&  ={\frac{\beta\,\left(  \hat{S}+m\right)  I}{1+\alpha\,I}}-\left(
p\delta+\gamma\right)  I-{\frac{\mathit{\beta_{2}}\,I}{1+\mathit{\alpha_{2}%
}\,I}}.%
\end{align}

Substituting $S=u-{\frac{\left(  \gamma+\beta_2 +bm\right)  v}{b}},I=v$
and $\beta m=p\delta+\gamma+\beta_{2}$ we obtain:%
\begin{align}
\frac{dv}{dt}  &  =0u+f(v,u)\nonumber\\
\frac{du}{dt}  &  =-bu+g(v,u), \label{eqst2}%
\end{align}
where
\begin{align}
&  f(u,v)=-{\frac{v\left(  -\beta\,b-\beta\,b\mathit{\alpha_{2}}\,v\right)
u}{\left(  1+\alpha\,v\right)  \left(  1+\mathit{\alpha_{2}}\,v\right)  b}%
}\nonumber\\
&  -\frac{v}{\left(  1+\alpha\,v\right)  \left(  1+\mathit{\alpha_{2}%
}\,v\right)  b}(\left(  \beta\,bm\mathit{\alpha_{2}}+b\gamma\,\alpha
\,\mathit{\alpha_{2}}-\beta\,\mathit{\alpha_{2}}\,p\delta+bp\delta
\,\alpha\,\mathit{\alpha_{2}}+{\beta}^{2}\mathit{\alpha_{2}}\,m\right)
{v}^{2}\nonumber\\
&  +\left(  bp\delta\,\mathit{\alpha_{2}}+\beta\,bm-\beta\,bm\mathit{alpha2}%
+b\gamma\,\mathit{\alpha_{2}}+{\beta}^{2}m-\beta\,p\delta+b\alpha
\,\beta\,m\right)  v),\nonumber\\
&  g(u,v)=-\frac{1}{\left(  1+\alpha\,v\right)  \left(  1+\mathit{\alpha_{2}%
}\,v\right)  b^{2}}[v((m{b}^{2}\gamma\,\alpha\,\mathit{\alpha_{2}}+2\,{\beta
}^{2}b{m}^{2}\mathit{\alpha_{2}}+\beta\,\mathit{\alpha_{2}}\,{p}^{2}{\delta
}^{2}\nonumber \\
&+{\beta}^{3}\mathit{\alpha_{2}}\,{m}^{2}-b\gamma\,p\delta\,\alpha
\,\mathit{\alpha_{2}}+b\gamma\,\alpha\,\mathit{\alpha_{2}}\,\beta\,m-2\,{\beta}^{2}%
\mathit{\alpha_{2}}\,mp\delta+bp\delta\,\alpha\,\mathit{\alpha_{2}}%
\,\beta\,m\nonumber \\
&-{b}^{2}m\alpha\,\mathit{\alpha_{2}}\,\beta-{\beta}^{2}%
bm\mathit{\alpha_{2}}+{b}^{2}{m}^{2}\beta\,\mathit{\alpha_{2}}+{b}^{2}%
mp\delta\,\alpha\,\mathit{\alpha_{2}}-b{p}^{2}{\delta}^{2}\alpha\,\mathit{\alpha_{2}}\nonumber \\
&-2\,\beta\,bm\mathit{\alpha_{2}}\,p\delta-{b}^{2}m\beta\,\mathit{\alpha_{2}}+b\beta\,\mathit{\alpha_{2}%
}\,p\delta)v^{2}+({b}^{2}{m}^{2}\beta-2\,\beta\,bmp\delta-\beta\,{b}%
^{2}m+2\,{\beta}^{2}b{m}^{2}\nonumber\\
&  -{\beta}^{2}bm+\beta\,{p}^{2}{\delta}^{2}+2\,\beta\,bm\mathit{\alpha_{2}%
}\,p\delta-bp\delta\,\alpha\,\beta\,m+{\beta}^{3}{m}^{2}+\beta\,bp\delta
-{b}^{2}\alpha\,\beta\,m-2\,{\beta}^{2}mp\delta\nonumber \\
&+{b}^{2}{m}^{2}\alpha\,\beta+b\alpha\,{\beta}^{2}{m}^{2}-u\beta\,{b}^{2}m\alpha_{{2}}-b{\beta}%
^{2}u\alpha_{{2}}m+b\beta\,u\alpha_{{2}}p\delta-\gamma\,bp\delta\,\alpha_{{2}%
}+\gamma\,\beta\,bm\alpha_{{2}}\nonumber \\
&+{b}^{2}mp\delta\,\alpha_{{2}}-{b}^{2}{m}%
^{2}\beta\,\alpha_{{2}}+u\beta\,{b}^{2}\alpha_{{2}}-b{p}^{2}{\delta}^{2}\alpha_{{2}}-{\beta}%
^{2}b{m}^{2}\alpha_{{2}}+{b}^{2}m\gamma\,\alpha_{{2}})v-{b}^{2}m\beta
\,u\nonumber \\
&+u\beta\,{b}^{2}-b{\beta}^{2}um+b\beta\,up\delta)].
\end{align}

By \cite{carr} the system (\ref{eqst2}) has a center manifold of the form
$u=h(v)$. Let $\phi:\mathbb{R}\rightarrow\mathbb{R}$ and define the
annihilator:
\begin{align}
N\phi &  =\phi^{\prime}(v)(f(v,\phi(v)))+b\phi-g(v,\phi(v))\nonumber\\
&  =\frac{1}{b^{2}(1+\alpha v)(1+\alpha_{2}v)}[bp\delta\,\alpha\,{v}^{3}%
\alpha_{{2}}\beta\,m+{b}^{2}{m}^{2}\beta\,{v}^{2}-\beta\,{v}^{2}{b}^{2}%
m+{b}^{3}\phi+{b}^{3}\phi\,\alpha\,v\nonumber \\
&+{b}^{3}\phi\,\alpha_{{2}}v+{b}^{2}m\gamma\,\alpha_{2}\,{v}^{2}+\phi\,\beta\,v{b}^{2}+vb\phi
\,\beta\,p\delta+{b}^{2}mp\delta\,\alpha_{{2}}{v}^{2}-\phi\,\beta\,{v}^{2}{b}^{2}m\alpha_{{2}}\nonumber \\
&-\gamma\,bp\delta\,\alpha_{{2}}{v}^{2}+{b}^{2}m\gamma\,\alpha\,{v}^{3}\alpha_{{2}}+\gamma\,\beta\,{v}^{2}%
bm\alpha_{{2}}+{b}^{2}{m}^{2}\alpha\,{v}^{2}\beta-2\,{\beta}^{2}{v}%
^{2}mp\delta-{\beta}^{2}{v}^{2}b{m}^{2}\alpha_{2}\nonumber \\
&+\beta\,{v}^{2}bp\delta\-{b}^{2}{v}^{2}\alpha\,\beta\,m+b\alpha\,{v}^{2}{\beta}^{2}{m}^{2}%
-bp\delta\,\alpha\,{v}^{2}\beta\,m+{\beta}^{3}{v}^{2}{m}^{2}-2\,\beta\,{v}%
^{2}bmp\delta+{\beta}^{3}{v}^{3}\alpha_{{2}}{m}^{2}\nonumber \\
&+2\,{\beta}^{2}{v}^{2}b{m}^{2}+\beta\,{v}^{2}{p}^{2}{\delta}^{2}-{\beta}^{2}{v}^{2}bm-\phi\,\beta
\,v{b}^{2}m-vb\phi\,{\beta}^{2}m+\beta\,{v}^{3}b\alpha_{{2}}p\delta-{b}^{2}%
{v}^{3}\alpha\,\alpha_{{2}}\beta\,m\nonumber \\
&-b{p}^{2}{\delta}^{2}\alpha\,{v}^{3}\alpha_{{2}}-2\,{\beta}^{2}{v}^{3}\alpha_{{2}}mp\delta+\phi\,\beta\,{v}^{2}{b}%
^{2}\alpha_{{2}}-{b}^{2}{m}^{2}\beta\,{v}^{2}\alpha_{{2}}+{b}^{2}{m}^{2}%
\beta\,{v}^{3}\alpha_{{2}}-\beta\,{v}^{3}{b}^{2}m\alpha_{{2}}\nonumber \\
&-{\beta}^{2}{v}^{3}b\alpha_{{2}}m+2\,{\beta}^{2}{v}^{3}b{m}^{2}\alpha_{{2}}-b{p}^{2}{\delta}^{2}{v}%
^{2}\alpha_{{2}}+\beta\,{v}^{3}\alpha_{{2}}{p}^{2}{\delta}^{2}+{b}^{3}%
\phi\,\alpha\,{v}^{2}\alpha_{{2}}-{v}^{2}b\phi\,{\beta}^{2}m\alpha_{{2}}\nonumber \\
&+{v}^{2}b\phi\,\beta\,p\delta\,\alpha_{{2}}+b\gamma\,\alpha\,{v}^{3}\alpha_{{2}}\beta\,m+{b}^{2}mp\delta\,\alpha
\,{v}^{3}\alpha_{{2}}-\gamma\,bp\delta\,\alpha\,{v}^{3}\alpha_{{2}}-2\,\beta\,{v}^{3}bm\alpha_{{2}}p\delta\nonumber \\
&+2\,\beta\,{v}^{2}bm\alpha_{{2}}p\delta].
\end{align}
Assume that $\phi=a_{0}v^{2}+a_{1}v^{3}+O(v^{4})$, then by substituting $\phi$
and $\frac{d\phi}{dv}$ in the annihilator $N\phi$ and expanding its Taylor
series we get:
\begin{align}
N\phi &  =\frac{1}{b^{2}}((\gamma\,\beta\,bm\alpha_{2}+{b}^{2}mp\delta
\,\alpha_{2}-{b}^{2}{m}^{2}\beta\,\mathit{\alpha_{2}}+2\,\beta
\,bm\mathit{\alpha_{2}}\,p\delta+2\,{\beta}^{2}b{m}^{2}+{b}^{2}{m}^{2}%
\beta\nonumber\\
&  -\beta\,{b}^{2}m+{b}^{3}\mathit{a_{0}}-{\beta}^{2}b{m}^{2}\mathit{\alpha
_{2}}-2\,\beta\,bmp\delta+{b}^{2}m\gamma\,\mathit{\alpha_{2}}-\gamma
\,bp\delta\,\mathit{\alpha_{2}}-{b}^{2}\alpha\,\beta\,m-{\beta}^{2}%
bm\nonumber\\
&  +b\alpha\,{\beta}^{2}{m}^{2}+{b}^{2}{m}^{2}\alpha\,\beta-2\,{\beta}%
^{2}mp\delta+\beta\,{p}^{2}{\delta}^{2}-bp\delta\,\alpha\,\beta\,m-b{p}%
^{2}{\delta}^{2}\mathit{\alpha_{2}}+\beta\,bp\delta\nonumber \\
&+{\beta}^{3}{m}^{2})v^{2}-\frac{1}{b^{2}}[\alpha\,{\beta}^{3}{m}^{2}-\mathit{a_{0}}\,\beta\,{b}%
^{2}-{b}^{3}\mathit{a_{1}}-2\,bp\delta\,\alpha\,\beta\,m-{\beta}^{2}b{m}%
^{2}{\mathit{\alpha_{2}}}^{2}-b{p}^{2}{\delta}^{2}{\mathit{\alpha_{2}}}^{2}\nonumber \\
&+m{b}^{2}\gamma\,{\mathit{\alpha_{2}}}^{2}-{b}^{2}{m}^{2}\beta
\,{\mathit{\alpha_{2}}}^{2}-{b}^{2}\alpha\,\beta\,m-{b}^{2}{\alpha}^{2}\beta\,m-\alpha\,{\beta}%
^{2}bm+b{\alpha}^{2}{\beta}^{2}{m}^{2}+{b}^{2}{m}^{2}{\alpha}^{2}\beta\nonumber \\
&+\alpha\,\beta\,{p}^{2}{\delta}^{2}+3\,\mathit{a_{0}}\,\beta\,{b}^{2}m+3\,\mathit{a_{0}}\,b{\beta}^{2}m+2\,\mathit{a_{0}}\,{b}^{2}\gamma
\,\mathit{\alpha_{2}}+2\,\mathit{a_{0}}\,{b}^{2}p\delta\,\mathit{\alpha_{2}%
}-2\,\mathit{a_{0}}\,\beta\,{b}^{2}m\mathit{\alpha_{2}}\nonumber \\
&-3\,\mathit{a_{0}}\,b\beta\,p\delta+2\,\mathit{a_{0}}\,{b}^{2}\alpha\,\beta\,m-2\,\alpha\,{\beta}^{2}mp\delta+\alpha\,\beta\,bp\delta+{b}^{2}%
mp\delta\,{\mathit{\alpha_{2}}}^{2}+b\gamma\,{\mathit{\alpha_{2}}}^{2}\beta\,m\nonumber \\
&-b\gamma\,p\delta\,{\mathit{\alpha_{2}}}^{2}+2\,\beta\,bm{\mathit{\alpha_{2}}}^{2}p\delta-bp\delta\,{\alpha}^{2}\beta\,m+{b}^{2}{m}^{2}\alpha\,
\beta+2\,b\alpha\,{\beta}^{2}{m}^{2}]v^{3}+O(v^{4})).
\end{align}
By choosing the coefficients of $v^{2}$ and $v^{3}$ in order to have
$N\phi=O(v^{4})$ we obtain that $a_{0}$ and $a_{1}$ must be the following:
\begin{align}
a_{0}  &  =-\frac{1}{b^{3}}[{b}^{2}{m}^{2}\beta+\beta\,bp\delta-{b}^{2}%
\alpha\,\beta\,m+{b}^{2}mp\delta\,\alpha_{2}-\gamma\,bp\delta\,\alpha
_{2}+\gamma\,\beta\,bm\alpha_{2}-2\,\beta\,bmp\delta\nonumber\\
&  -\beta\,{b}^{2}m+2\,{\beta}^{2}b{m}^{2}-{\beta}^{2}bm+\beta\,{p}^{2}%
{\delta}^{2}+2\,\beta\,bm\alpha_{2}\,p\delta-bp\delta\,\alpha\,\beta
\,m+{\beta}^{3}{m}^{2}-b{p}^{2}{\delta}^{2}\alpha_{2}\nonumber\\
&  -{b}^{2}{m}^{2}\beta\,\alpha_{2}-2\,{\beta}^{2}mp\delta+{b}^{2}%
m\gamma\,\alpha_{2}+b\alpha\,{\beta}^{2}{m}^{2}+{b}^{2}{m}^{2}\alpha
\,\beta-{\beta}^{2}b{m}^{2}\alpha_{2}],\\
a_{1}  &  =\frac{1}{b^{3}}[\alpha\,{\beta}^{3}{m}^{2}-ao\,\beta\,{b}%
^{2}-2\,bp\delta\,\alpha\,\beta\,m-{\beta}^{2}b{m}^{2}\alpha_{2}-b{p}%
^{2}{\delta}^{2}{\mathit{\alpha_{2}}}^{2}+m{b}^{2}\gamma\,\alpha_{2}^{2}%
-{b}^{2}{m}^{2}\beta\alpha_{2}^{2}\nonumber\\
&  -{b}^{2}\alpha\,\beta\,m-{b}^{2}{\alpha}^{2}\beta\,m-\alpha\,{\beta}%
^{2}bm+b{\alpha}^{2}{\beta}^{2}{m}^{2}+{b}^{2}{m}^{2}{\alpha}^{2}\beta
+\alpha\,\beta\,{p}^{2}{\delta}^{2}+3\,\mathit{a_{0}}\,\beta\,{b}%
^{2}m\nonumber\\
&  +3\,a_{0}\,b{\beta}^{2}m+2\,a_{0}\,{b}^{2}\gamma\,\alpha_{2}%
+2\,\mathit{a_{0}}\,{b}^{2}p\delta\,\alpha_{2}-2\,\mathit{a_{0}}\,\beta
\,{b}^{2}m\mathit{\alpha_{2}}-3\,\mathit{a_{0}}\,b\beta\,p\delta
+2\,\mathit{a_{0}}\,{b}^{2}\alpha\,\beta\,m\nonumber\\
&  -2\,\alpha\,{\beta}^{2}mp\delta+\alpha\,\beta\,bp\delta+{b}^{2}%
mp\delta\,{\mathit{\alpha_{2}}}^{2}+b\gamma\,{\mathit{\alpha_{2}}}^{2}%
\beta\,m-b\gamma\,p\delta\,{\mathit{\alpha_{2}}}^{2}+2\,\beta
\,bm{\mathit{\alpha_{2}}}^{2}p\delta\nonumber\\
&-bp\delta\,{\alpha}^{2}\beta\,m+{b}^{2}{m}^{2}\alpha\,\beta+2\,b\alpha\,{\beta}^{2}{m}^{2}].
\end{align}
Hence $h(v)=a_{0}v^{2}+a_{1}v^{3}+O(v^{4})$.

\section{Hopf bifurcation}

To analyze the behaviour of the solutions of \eqref{ruanmod} when $  s=0 $ we make a change of coordinates $x=S-S_2$, $y=I-I_2$, to obtain a new equivalent system to \eqref{ruanmod} with an equilibrium in $(0,0)$ in the $x-y$ plane. Under this change the system becomes in:
\begin{align}
\frac{dx}{dt} &= {\frac {a_{{11}}x+a_{{12}}y+c_{{1}}xy+c_{{2}}{y}^{2} + c_7}{1+\alpha_{{1}}y+
\alpha_{{1}}I_{{2}}}}, \nonumber \\
\frac{dy}{dt} &= {\frac {a_{{21}}x+a_{{22}}y+c_{{3}}xy+c_{{4}}x{y}^{2}+c_{{5}}{y}^{2}+c
_{{6}}{y}^{3} + c_8}{ \left( 1+\alpha_{{1}}y+\alpha_{{1}}I_{{2}} \right)
 \left( 1+\alpha_{{2}}y+\alpha_{{2}}I_{{2}} \right) }}.  \label{hmrhb2}
\end{align}
Where:
\begin{align}
a_{11} &= -b-\beta_{{1}}I_{{2}}-b\alpha_{{1}}I_{{2}} \\
a_{12} &= -2\,bm\alpha_{{1}}I_{{2}}+bm\alpha_{{1}}-b\alpha_{{1}}S_{{2}}+2\,p
\delta\,\alpha_{{1}}I_{{2}}+p\delta-bm-\beta_{{1}}S_{{2}} \\
c_{1} &=-b\alpha_{{1}}-\beta_{{1}} \\
c_{2} &= -bm\alpha_{{1}}+p\delta\,\alpha_{{1}} \\
a_{21} &=-I_{{2}} \left( -\beta_{{1}}-\beta_{{1}}\alpha_{{2}}I_{{2}} \right) \\
a_{22} &= -2\,p\delta\,\alpha_{{1}}I_{{2}}+2\,\beta_{{1}}\alpha_{{2}}S_{{2}}I_{{
2}}-3\,p\delta\,\alpha_{{1}}\alpha_{{2}}{I_{{2}}}^{2}-2\,\gamma\,
\alpha_{{1}}I_{{2}} \nonumber \\
&-2\,\gamma\,\alpha_{{2}}I_{{2}}-2\,p\delta\,\alpha_
{{2}}I_{{2}}-2\,\beta_{{2}}\alpha_{{1}}I_{{2}}-3\,\gamma\,\alpha_{{1}}
\alpha_{{2}}{I_{{2}}}^{2}-\gamma-p\delta-\beta_{{2}}+\beta_{{1}}S_{{2}
} \\
c_{3} &= 2\,\beta_{{1}}\alpha_{{2}}I_{{2}}+\beta_{{1}} \\
c_{4} &= \beta_{{1}}\alpha_{{2}}{y}^{2} \\
c_{5} &= -3\,p\delta\,\alpha_{{1}}\alpha_{{2}}I_{{2}}-3\,\gamma\,\alpha_{{1}}
\alpha_{{2}}I_{{2}}-p\delta\,\alpha_{{1}}+\beta_{{1}}\alpha_{{2}}S_{{2
}}-\gamma\,\alpha_{{1}}-\gamma\,\alpha_{{2}}-p\delta\,\alpha_{{2}}\nonumber \\
&-\beta_{{2}}\alpha_{{1}} \\
c_{6} &= -p\delta\,\alpha_{{1}}\alpha_{{2}}-\gamma\,\alpha_{{1}}\alpha_{{2}} \\
c_7 &= -(\beta_{{1}}S_{{2}}I_{{2}}-bm\alpha_{{1}}I_{{2}}+bS_{{2}}-p\delta\,I_{{
2}}-p\delta\,\alpha_{{1}}{I_{{2}}}^{2}+b\alpha_{{1}}S_{{2}}I_{{2}}+bmI
_{{2}} \nonumber \\
&-bm+bm\alpha_{{1}}{I_{{2}}}^{2})\\
c_8 &=- I_2 [p\delta\,\alpha_{{1}}I_{{2}}+p\delta+p\delta\,\alpha_{{2}}I_{{2}}+
\gamma\,\alpha_{{2}}I_{{2}}-\beta_{{1}}\alpha_{{2}}S_{{2}}I_{{2}}+
\gamma\,\alpha_{{1}}I_{{2}}+\beta_{{2}}\alpha_{{1}}I_{{2}} \nonumber \\
&+\gamma+ \gamma\,\alpha_{{1}}\alpha_{{2}}{I_{{2}}}^{2}-\beta_{{1}}S_{{2}}+\beta
_{{2}}+p\delta\,\alpha_{{1}}\alpha_{{2}}{I_{{2}}}^{2}].
\end{align}
But from the equations for the equilibrium point we can prove that $c_7=c_8=0$, so the system we will work on is
\begin{align}
\frac{dx}{dt} &= {\frac {a_{{11}}x+a_{{12}}y+c_{{1}}xy+c_{{2}}{y}^{2} }{1+\alpha_{{1}}y+
\alpha_{{1}}I_{{2}}}}, \nonumber \\
\frac{dy}{dt} &= {\frac {a_{{21}}x+a_{{22}}y+c_{{3}}xy+c_{{4}}x{y}^{2}+c_{{5}}{y}^{2}+c
_{{6}}{y}^{3} }{ \left( 1+\alpha_{{1}}y+\alpha_{{1}}I_{{2}} \right)
 \left( 1+\alpha_{{2}}y+\alpha_{{2}}I_{{2}} \right) }}.
\end{align}
If we denote system \eqref{ruanmod} as $ (S,I)' = f(S,I) $ and system \eqref{hmrhb2} as $(x,y)' = F(x,y)$, $ f=(f_1,f_2) $, $F=(F_1,F_2)$ then
$$ F(x,y) = f(x+S_2,y+I_2),  $$
and
$$  \frac{ \partial F_i }{  \partial x } (x,y) = \frac{ \partial f_i }{ \partial S } (x+S_2,y+I_2)\frac{ \partial S }{ \partial x} (x,y) + \frac{ \partial f_i }{ \partial I } (x+S_2,y+I_2)\frac{ \partial I }{ \partial x} (x,y)=  \frac{ \partial f_i }{ \partial S } (x+S_2,y+I_2)$$
$$ \frac{ \partial F_i }{  \partial y } (x,y) = \frac{ \partial f_i }{ \partial S } (x+S_2,y+I_2)\frac{ \partial S }{ \partial y} (x,y) + \frac{ \partial f_i }{ \partial I } (x+S_2,y+I_2)\frac{ \partial I }{ \partial y} (x,y)=  \frac{ \partial f_i }{ \partial S } (x+S_2,y+I_2) .$$
So, the jacobian matrix  of \eqref{hmrhb1} $ DF(0,0) $ in the equilibrium is equal to the jacobian matrix of system \eqref{ruanmod} $Df(S_1,I_1)$. We can also compute the partial derivatives of system \eqref{hmrhb2} and \eqref{hmrhb1} to prove that they are equal,ie,
\begin{equation}
Df(S_2,I_2) = DF(0,0).
\end{equation}
 Therefore the system \eqref{hmrhb1} and \eqref{ruanmod} are equivalent and we can work with system \eqref{hmrhb1}. The jacobian matrix $DF(0,0)$ of \eqref{hmrhb1} is:
\begin{equation}
DF(0,0)=  \left[ \begin {array}{cc} {\dfrac {a_{{11}}}{1+\alpha_{{1}}I_{{2}}}}&
{\dfrac {a_{{12}}}{ \left( 1+\alpha_{{1}}I
_{{2}} \right) }}\\ \noalign{\medskip}{\dfrac {a_{{21}}}{ \left( 1+
\alpha_{{2}}I_{{2}} \right)  \left( 1+\alpha_{{1}}I_{{2}} \right) }}&{
\dfrac {a_{{22}}}{ \left( 1+\alpha_{{2}}I_{{2}} \right)  \left( 1+
\alpha_{{1}}I_{{2}} \right) }}\end {array} \right].
\end{equation}

  So system \eqref{hmrhb1} can be rewritten as

\begin{align}
\dfrac{dx}{dt} &= {\dfrac {a_{{11}}x}{1+\alpha_{{1}}I_{{2}}}}+{\frac {a_{{12}}y}{1+\alpha
_{{1}}I_{{2}}}} + G_1 (x,y) \\
\dfrac{dy}{dt} &= {\frac {a_{{21}}x}{ \left( 1+\alpha_{{1}}I_{{2}} \right)  \left( 1+
\alpha_{{2}}I_{{2}} \right) }}+{\frac {a_{{22}}y}{ \left( 1+\alpha_{{1
}}I_{{2}} \right)  \left( 1+\alpha_{{2}}I_{{2}} \right) }} + G_2(x,y).
\end{align}

Where

\begin{align}
G_1 &= \dfrac{1}{(1 + \alpha_1 y+ \alpha_1 I_2)(1 + \alpha_1 I_2)} \{[(1 + \alpha_1 I_2)c_1-a_{11} \alpha_1] xy + [ c_2 ( 1 + \alpha_1 I_2 ) - \alpha_1 a_{12} ] y^{2} \} \\
G_2 &= \dfrac{1}{(1 + \alpha_1 y+ \alpha_1 I_2)(1 + \alpha_2 y+ \alpha_2 I_2)(1 + \alpha_1 I_2)(1 + \alpha_2 I_2)} \{ [c_3 (1 + \alpha_1 I_2)(1 + \alpha_2 I_2)\nonumber \\ &- a_{21} ( \alpha_2 + \alpha_1 + 2 \alpha_1 \alpha_2 I_2 ) ] x y + [ c_4 ( 1 + \alpha_1 I_2 )( 1 + \alpha_2 I_2 ) - a_{21} \alpha_1 \alpha_2 ] xy^{2}\nonumber \\ &+ [c_{5}( 1 + \alpha_1 I_2 )( 1 + \alpha_2 I_2) - a_{22} ( \alpha_2 + \alpha_1 + 2 \alpha_1 \alpha_2 I_1 )  ] y^{2} + [ c_{6} ( 1 + \alpha_1 I_2 )( 1 + \alpha_2 I_2 )\nonumber \\
&- a_{22} \alpha_1 \alpha_2 ] y^{3} \}.
\end{align}

We need the normal form of the system \eqref{hmrhb1}. The eigenvalues of $ DF(0,0) $ when $s_2=0$ and (i),(ii) are satisfied are:
$$   \Lambda i , - \Lambda i .$$
With complex eigenvector $$ v= \left( \begin{matrix}
- 1 \\ \dfrac{ - \Lambda i(1 + \alpha_1 I_2)+ a_{11} }{a_{12}}
 \end{matrix}\right), \quad \bar{v} = \left( \begin{matrix}
- 1 \\ \dfrac{  \Lambda i(1 + \alpha_1 I_2)+ a_{11} }{a_{12}}
 \end{matrix}\right) .$$
 Using the Jordan Canonical form of matrix $DF(0,0)$ and the procedure in \cite{perko} (p. 107, 108) we use the change of variable $u=x, v= \frac{a_{11}}{ \Lambda( 1+ \alpha_1  I_2)} + \frac{a_{12} y}{  \Lambda(1 + \alpha_{1} I_2 ) } $, to obtain the following equivalent system:

\begin{equation}
\left( \begin{matrix}
 u \\ v
 \end{matrix}\right)= \left( \begin{matrix}
 0 & \Lambda \\ - \Lambda & 0
 \end{matrix}\right) \left( \begin{matrix}
 u \\ v
 \end{matrix}\right) + \left( \begin{matrix}
 H_1 (u,v) \\ H_2(u,v)
 \end{matrix}\right).
\end{equation}
Where
\begin{align}
H_1(u,v) &= -\dfrac{\left(  \left( -a_{{12}}c_{{1}}+a_{{11}}c_{{2}} \right) u+ \left( -
  \Lambda c_{{2}}\alpha_{{1}}I_{{2}}+  \Lambda a_{{12}
}\alpha_{{1}}-  \Lambda c_{{2}} \right) v \right)  \left(
 \left(  \Lambda+  \Lambda \alpha_{{1}}I_{{2}} \right) v-a_{{11}}u
 \right)
}{a_{{12}} \left(  \left( \alpha_{{1}}  \Lambda +  \Lambda {\alpha_{{1}}}^{2}I_{{2}} \right) v+a_{{12}}-\alpha_{{1}}a_{{
11}}u+a_{{12}}\alpha_{{1}}I_{{2}} \right)
}  \\
H_2(u,v) &= - \dfrac{1}{h(u,v)} \left[ (\Lambda(1+ \alpha_1 I_2)v-a_{11}u) \left( A_1 v^{2}+A_2 uv + A_3 v + A_4 u^{2} + A_5 u \right) \right].
\end{align}
   And:

   \begin{align*}
A_1 &= {\Lambda}^{2} \left( 1+\alpha_{{1}}I_{{2}} \right) ^{2} [ -a_{{12
}}c_{{6}}\alpha_{{2}}{I_{{2}}}^{2}\alpha_{{1}}-a_{{11}}c_{{2}}\alpha_{
{1}}{I_{{2}}}^{2}{\alpha_{{2}}}^{2}-a_{{11}}c_{{2}}\alpha_{{1}}I_{{2}}
\alpha_{{2}}\nonumber \\ 
&-a_{{12}}c_{{6}}\alpha_{{1}}I_{{2}}+a_{{11}}a_{{12}}\alpha
_{{1}}{\alpha_{{2}}}^{2}I_{{2}}+a_{{11}}a_{{12}}\alpha_{{1}}\alpha_{{2
}} +a_{{12}}a_{{22}}\alpha_{{1}}\alpha_{{2}}\nonumber \\
&-a_{{11}}c_{{2}}{\alpha_{{2
}}}^{2}I_{{2}}-a_{{12}}c_{{6}}\alpha_{{2}}I_{{2}}-a_{{11}}c_{{2}}
\alpha_{{2}}-a_{{12}}c_{{6}} ]
 \\
A_2 &= -\Lambda\, \left( 1+\alpha_{{1}}I_{{2}} \right)  [ a_{{11}}a_{{12
}}c_{{1}}{\alpha_{{2}}}^{2}\alpha_{{1}}{I_{{2}}}^{2}+{a_{{12}}}^{2}c_{
{4}}\alpha_{{2}}{I_{{2}}}^{2}\alpha_{{1}}-2\,a_{{12}}a_{{11}}c_{{6}}
\alpha_{{2}}{I_{{2}}}^{2}\alpha_{{1}}\nonumber \\
&-2\,c_{{2}}\alpha_{{1}}{I_{{2}}}^{2}{\alpha_{{2}}}^{2}{a_{{11}}}^{2} +a_{{12}}\alpha_{{1}}{\alpha_{{2}}}^{2}{a_{{11}}}^{2}I_{{2}}-2\,a_{{12}}a_{{11}}c_{{6}}\alpha_{{1}}I_{{2}
}-2\,c_{{2}}\alpha_{{1}}I_{{2}}\alpha_{{2}}{a_{{11}}}^{2}\nonumber \\
&+{a_{{12}}}^{2}c_{{4}}\alpha_{{1}}I_{{2}}+a_{{11}}a_{{12}}c_{{1}}\alpha_{{2}}\alpha
_{{1}}I_{{2}}+a_{{11}}a_{{12}}c_{{1}}{\alpha_{{2}}}^{2}I_{{2}}+{a_{{12
}}}^{2}c_{{4}}\alpha_{{2}}I_{{2}}-2\,a_{{12}}a_{{11}}c_{{6}}\alpha_{{2
}}I_{{2}}\nonumber \\
&-2\,{a_{{11}}}^{2}c_{{2}}{\alpha_{{2}}}^{2}I_{{2}}+{a_{{12}}}
^{2}c_{{4}}-2\,{a_{{11}}}^{2}c_{{2}}\alpha_{{2}}-2\,a_{{12}}a_{{11}}c_{{6}}+a_{{11}}a_{{12}}c_{{1}}\alpha_{{2}}+a_{{12}}\alpha_{{1}}\alpha_{
{2}}{a_{{11}}}^{2}\nonumber \\
&-{a_{{12}}}^{2}\alpha_{{1}}a_{{21}}\alpha_{{2}}+2\,a_{{12}}a_{{11}}a_{{22}}\alpha_{{1}}\alpha_{{2}} ] \\
A_3 &= \Lambda\, \left( 1+\alpha_{{1}}I_{{2}} \right) a_{{12}} [ -a_{{12
}}c_{{5}}\alpha_{{1}}{I_{{2}}}^{2}\alpha_{{2}}+a_{{12}}a_{{11}}\alpha_
{{1}}{\alpha_{{2}}}^{2}{I_{{2}}}^{2}+2\,a_{{12}}a_{{22}}\alpha_{{1}}
\alpha_{{2}}I_{{2}}\nonumber \\
&+2\,a_{{12}}a_{{11}}\alpha_{{1}}\alpha_{{2}}I_{{2}}
-a_{{12}}c_{{5}}\alpha_{{1}}I_{{2}}+a_{{12}}a_{{22}}\alpha_{{1}}+a_{{
11}}a_{{12}}\alpha_{{1}}-a_{{12}}c_{{5}}\alpha_{{2}}I_{{2}}+a_{{12}}a_
{{22}}\alpha_{{2}}\nonumber \\
&-2\,a_{{11}}c_{{2}}\alpha_{{1}}{I_{{2}}}^{2}\alpha_{{2}}-a_{{11}}c_{{2}}\alpha_{{1}}I_{{2}}-a_{{11}}c_{{2}}
-a_{{11}}c_{{2}}{\alpha_{{2}}}^{2}{I_{{2}}}^{2}-2\,a_{{11}}c_{{2}}\alpha_{{2}}I_{{2}} ] \\
A_4 &= -a_{{11}} [-{a_{{12}}}^{2}c_{{4}}\alpha_{{2}}I_{{2}}-a_{{11}}a_{
{12}}c_{{1}}\alpha_{{2}}\alpha_{{1}}I_{{2}}+c_{{2}}\alpha_{{1}}I_{{2}}
\alpha_{{2}}{a_{{11}}}^{2}-{a_{{12}}}^{2}c_{{4}}\alpha_{{2}}{I_{{2}}}^
{2}\alpha_{{1}}\nonumber \\
&-{a_{{12}}}^{2}c_{{4}}\alpha_{{1}}I_{{2}}-{a_{{12}}}^{2
}c_{{4}}+{a_{{12}}}^{2}\alpha_{{1}}a_{{21}}\alpha_{{2}}-a_{{12}}a_{{11
}}a_{{22}}\alpha_{{1}}\alpha_{{2}}-a_{{11}}a_{{12}}c_{{1}}{\alpha_{{2}
}}^{2}\alpha_{{1}}{I_{{2}}}^{2}\nonumber \\
&+a_{{12}}a_{{11}}c_{{6}}+{a_{{11}}}^{2}
c_{{2}}{\alpha_{{2}}}^{2}I_{{2}}+a_{{12}}a_{{11}}c_{{6}}\alpha_{{2}}{I
_{{2}}}^{2}\alpha_{{1}}+{a_{{11}}}^{2}c_{{2}}\alpha_{{2}}-a_{{11}}a_{{
12}}c_{{1}}\alpha_{{2}}\nonumber \\
&+a_{{12}}a_{{11}}c_{{6}}\alpha_{{2}}I_{{2}}+a_{{12}}a_{{11}}c_{{6}}\alpha_{{1}}I_{{2}}+c_{{2}}\alpha_{{1}}{I_{{2}}}^{
2}{\alpha_{{2}}}^{2}{a_{{11}}}^{2}-a_{{11}}a_{{12}}c_{{1}}{\alpha_{{2}}}^{2}I_{{2}} ] \\
A_5 &= a_{{12}} [ 2\,{a_{{11}}}^{2}c_{{2}}\alpha_{{2}}I_{{2}}+{a_{{11}}}
^{2}c_{{2}}\alpha_{{1}}I_{{2}}+{a_{{12}}}^{2}\alpha_{{1}}a_{{21}}+{a_{
{11}}}^{2}c_{{2}}{\alpha_{{2}}}^{2}{I_{{2}}}^{2}-a_{{12}}a_{{11}}a_{{
22}}\alpha_{{2}}\nonumber \\
&-a_{{12}}a_{{11}}a_{{22}}\alpha_{{1}}-{a_{{12}}}^{2}c_{{3}}\alpha_{{2}}I_{{2}}-{a_{{12}}}^{2}c_{{3}}\alpha_{{1}}I_{{2}}+a_{{
11}}a_{{12}}c_{{5}}+{a_{{12}}}^{2}\alpha_{{2}}a_{{21}}-a_{{12}}a_{{11}
}c_{{1}}\nonumber \\
&+2\,{a_{{12}}}^{2}\alpha_{{1}}a_{{21}}\alpha_{{2}}I_{{2}}-a_{{
12}}a_{{11}}c_{{1}}{\alpha_{{2}}}^{2}{I_{{2}}}^{2}-a_{{12}}a_{{11}}c_{
{1}}\alpha_{{1}}I_{{2}}-2\,a_{{12}}a_{{11}}c_{{1}}\alpha_{{2}}I_{{2}}\nonumber \\
&+a_{{11}}a_{{12}}c_{{5}}\alpha_{{2}}I_{{2}}+a_{{11}}a_{{12}}c_{{5}}
\alpha_{{1}}I_{{2}}-{a_{{12}}}^{2}c_{{3}}\alpha_{{1}}{I_{{2}}}^{2}
\alpha_{{2}}-{a_{{12}}}^{2}c_{{3}}+a_{{11}}a_{{12}}c_{{5}}\alpha_{{1}}
{I_{{2}}}^{2}\alpha_{{2}}\nonumber \\
&-2\,a_{{12}}a_{{11}}a_{{22}}\alpha_{{1}}
\alpha_{{2}}I_{{2}}-2\,a_{{12}}a_{{11}}c_{{1}}\alpha_{{1}}{I_{{2}}}^{2
}\alpha_{{2}}-a_{{12}}a_{{11}}c_{{1}}\alpha_{{1}}{I_{{2}}}^{3}{\alpha_
{{2}}}^{2}+2\,{a_{{11}}}^{2}c_{{2}}\alpha_{{1}}{I_{{2}}}^{2}\alpha_{{2
}}\nonumber \\
&+{a_{{11}}}^{2}c_{{2}}\alpha_{{1}}{I_{{2}}}^{3}{\alpha_{{2}}}^{2}+{a
_{{11}}}^{2}c_{{2}} ]
 \\
h(u,v) &=\Lambda\, \left( 1+\alpha_{{1}}I_{{2}} \right) ^{2}a_{{12}} [
 \left( \alpha_{{1}}\Lambda+\Lambda\,{\alpha_{{1}}}^{2}I_{{2}}
 \right) v+a_{{12}}-\alpha_{{1}}a_{{11}}u+a_{{12}}\alpha_{{1}}I_{{2}}
 ] \nonumber \\
 & [  \left( \alpha_{{2}}\Lambda+\alpha_{{2}}\Lambda\,
\alpha_{{1}}I_{{2}} \right) v+a_{{12}}-\alpha_{{2}}a_{{11}}u+\alpha_{{
2}}I_{{2}}a_{{12}} ]  \left( 1+\alpha_{{2}}I_{{2}} \right) .
\end{align*}

Let
\begin{align}
\bar{a}_2 &= \dfrac{1}{16} [ (H_1)_{uuu} + (H_1)_{uvv} + (H_{2})_{uuv} + (H_2)_{vvv}  ] + \dfrac{1}{16( - \Lambda)}  [(H_1)_{uv} ((H_1)_{uu} + (H_1)_{vv}) \nonumber \\
& - (H_{2})_{uv} ( (H_{2})_{uu} + (H_2)_{vv} ) - (H_1)_{uu}(H_2)_{uu} + (H_1)_{vv} (H_2)_{vv}].
\end{align}
 Then
\begin{align}
\bar{a}_{2} &= \dfrac{3\left(  \left( -c_{{1}}\Lambda\,v{\alpha_{{1}}}^{2}I_{{2}}+\Lambda\,va_{{11}}
{\alpha_{{1}}}^{2}-a_{{12}}c_{{1}}\alpha_{{1}}I_{{2}}+a_{{11}}c_{{2}}
\alpha_{{1}}I_{{2}}-c_{{1}}\Lambda\,v\alpha_{{1}}-a_{{12}}c_{{1}}+a_{{
11}}c_{{2}} \right) a_{{12}}{a_{{11}}}^{2}\alpha_{{1}}
 \right)}{8 \left( a_{{12}}+\alpha_{{1}}\Lambda\,v \right) ^{4} \left( 1+\alpha_{
{1}}I_{{2}} \right) ^{3}
}  \nonumber \\
& - \dfrac{\left( -3\,a_{{11}}c_{{2}}-3\,a_{{11}}c_{{2}}\alpha_{{1}}I_{{2}}+2\,a
_{{12}}a_{{11}}\alpha_{{1}}+a_{{12}}c_{{1}}+a_{{12}}c_{{1}}\alpha_{{1}
}I_{{2}} \right) \alpha_{{1}}{\Lambda}^{2}
}{8 \left( 1+\alpha_{{1}}I_{{2}} \right) {a_{{12}}}^{3}} \nonumber \\
& - \dfrac{1}{8 \Lambda\, \left( 1+\alpha_{{1}}I_{{2}} \right) ^{4}{a_{{12}}}^{4}
 \left( 1+\alpha_{{2}}I_{{2}} \right) ^{3}
}[ 2\,a_{{11}}A_{{5}}\alpha_{{1}}\Lambda+6\,a_{{11}}A_{{5}}\alpha_{{1}}
\Lambda\,\alpha_{{2}}I_{{2}}+2\,a_{{11}}A_{{5}}{\alpha_{{1}}}^{2}
\Lambda\,I_{{2}} \nonumber \\
&+4\,a_{{11}}A_{{5}}{\alpha_{{1}}}^{2}\Lambda\,{I_{{2}}
}^{2}\alpha_{{2}}+2\,a_{{11}}A_{{5}}\alpha_{{2}}\Lambda-{a_{{11}}}^{2}
A_{{3}}\alpha_{{1}}-2\,{a_{{11}}}^{2}A_{{3}}\alpha_{{1}}\alpha_{{2}}I_
{{2}}-{a_{{11}}}^{2}A_{{3}}\alpha_{{2}}-a_{{11}}A_{{2}}a_{{12}} \nonumber \\
&-a_{{11
}}A_{{2}}a_{{12}}\alpha_{{2}}I_{{2}}-a_{{11}}A_{{2}}a_{{12}}\alpha_{{1
}}I_{{2}}-a_{{11}}A_{{2}}a_{{12}}\alpha_{{1}}{I_{{2}}}^{2}\alpha_{{2}}
+A_{{4}}\Lambda\,a_{{12}}+A_{{4}}\Lambda\,a_{{12}}\alpha_{{2}}I_{{2}}\nonumber \\
&+2\,A_{{4}}\Lambda\,a_{{12}}\alpha_{{1}}I_{{2}}+2\,A_{{4}}\Lambda\,a_{{
12}}\alpha_{{1}}{I_{{2}}}^{2}\alpha_{{2}}+A_{{4}}\Lambda\,a_{{12}}{
\alpha_{{1}}}^{2}{I_{{2}}}^{2}+A_{{4}}\Lambda\,a_{{12}}{\alpha_{{1}}}^
{2}{I_{{2}}}^{3}\alpha_{{2}}]\nonumber \\
&+ \dfrac{3}{8} {\frac {(-A_{{1}}a_{{12}}-A_{{1}}a_{{12}}\alpha_{{2}}I_{{2}}+A_{{3}}
\alpha_{{1}}\Lambda+2\,A_{{3}}\alpha_{{1}}\Lambda\,\alpha_{{2}}I_{{2}}
+A_{{3}}\alpha_{{2}}\Lambda)}{ \left( 1+\alpha_{{1}}I_{{2}} \right) ^{2
}{a_{{12}}}^{4} \left( 1+\alpha_{{2}}I_{{2}} \right) ^{3}}} \nonumber \\
 &- \dfrac{1}{16 \Lambda }[ -2\,{\frac {\Lambda\, \left( -2\,a_{{11}}c_{{2}}-2\,a_{{11}}c_{{2}}
\alpha_{{1}}I_{{2}}+a_{{12}}a_{{11}}\alpha_{{1}}+a_{{12}}c_{{1}}+a_{{
12}}c_{{1}}\alpha_{{1}}I_{{2}} \right)   }{{a_{{12}}}^{4} \left( 1+\alpha_{{1}}I_{{2}}
 \right) ^{2}}} \nonumber \\
 & -2\,{\frac { \left( A_{{5}}\Lambda+A_{{5}}\Lambda\,\alpha_{{1}}I_{{2}}
-a_{{11}}A_{{3}} \right)  \left( -a_{{11}}A_{{5}}+A_{{3}}\Lambda+A_{{3
}}\Lambda\,\alpha_{{1}}I_{{2}} \right) }{{\Lambda}^{2} \left( 1+\alpha
_{{1}}I_{{2}} \right) ^{6}{a_{{12}}}^{6} \left( 1+\alpha_{{2}}I_{{2}}
 \right) ^{4}}}  \nonumber \\
 & -4\,{\frac { \left( -a_{{12}}c_{{1}}+a_{{11}}c_{{2}} \right) {a_{{11}}
}^{2}A_{{5}}}{{a_{{12}}}^{5} \left( 1+\alpha_{{1}}I_{{2}} \right) ^{4}
\Lambda\, \left( 1+\alpha_{{2}}I_{{2}} \right) ^{2}}}
4\,{\frac { \left( -c_{{2}}\alpha_{{1}}I_{{2}}+a_{{12}}\alpha_{{1}}-c_
{{2}} \right) {\Lambda}^{2}A_{{3}}}{{a_{{12}}}^{5} \left( 1+\alpha_{{1
}}I_{{2}} \right) ^{2} \left( 1+\alpha_{{2}}I_{{2}} \right) ^{2}}}
].
\end{align}


\begin{thebibliography}{99}                                                                                               %


\bibitem {hadeler1997}Hadeler, K.P., van den Driessche, P. Backward
bifurcation in epidemic control. Math. Biosci. 146, 15--35 (1997).

\bibitem {dushoff1998}J. Dushoff, W. Huang, C. Castillo-Chavez, Backwards
bifurcations and catastrophe in simple models of fatal diseases, J. Math.
Biol. 36 227--248, (1998).

\bibitem {driessche2000}P. van den Driessche, J. Watmough, A simple SIS
epidemic model with a backward bifurcation, J. Math. Biol. 40 525--540 (2000).

\bibitem {brauer2004}Brauer, F.: Backward bifurcations in simple vaccination
models. J. Math. Anal. Appl. 298, 418--431 (2004)

\bibitem {hadeler1995}K.P. Hadeler, C. Castillo-Chavez, A core group model for
disease transmission, Math. Biosci. 128 (1995) 41.

\bibitem {anderson1991}R.M. Anderson, R.M. May, Infectious Diseases of Humans,
Oxford University Press, London, 1991.

\bibitem {wang2004}W. Wang, S. Ruan, Bifurcation in an epidemic model with
constant removal rate of the infectives, J. Math. Anal. Appl. 291 (2004) 775.

\bibitem {humaruan}Zhixing Hu, Wanbiao Ma, Shigui Ruan, \textit{Analysis of
SIR epidemic models with nonlinear incidence rate and treatment}, Mathematical
Biosciences, 238 (2012) 12-20.

\bibitem {zhonghua}Zhang Zhonghua, Suo Yaohong, \textit{Qualitative analysis
of a SIR epidemic model with saturated treatment rate}, J Appl Math Comput
(2010) 34: 177-194.

\bibitem {likuang}Bingtuan Li, Yang Kuang, \textit{Simple Food Chain in a
Chemostat with Distinct. Removal Rates},Journal of Mathematical Analysis and
Applications 242, 75-92(2000).

\bibitem {carr}Jack Carr, \textit{Applications of Centre Manifold Theory},
(1981), pp 1-13

\bibitem {perko}Lawrence Perko, \textit{Differential Equations and Dynamical
Systems}, 3rd edition, Springer,107-108.

\bibitem {zhoufan}Linhua Zhou, Meng Fan, \textit{Dynamics of an SIR epidemic
model with limited medical resources revisited}, Nonlinear Analysis: Real
World Application 13(2012)312-324.

\bibitem {kelley}Walter G. Kelley, Allan C. Peterson, \textit{The Theory of
Differential Equations, Classical and Qualitative}, second edition, Springer


\bibitem {zhangliu} Xu Zhang, Xianning Liu, \textit{Backward bifurcation of an epidemic model with saturated treatment function}, J. Math. Anal. Appl. 348 (2008) 433-443

\bibitem {wancui} H. Wan J. Cui \textit{Rich Dynamics of an epidemic model with saturation recovery}, J. App. Math. v.2013, Article ID 314958.

\bibitem{guckenheimer} John Guckenheimer, Philip Holmes, \textit{Nonlinear Oscillations, Dynamical Systems, and Bifurcation of Vector Fields}, Springer-Verlag, New York, 1996.
\end{thebibliography}
\end{document}